\documentclass[11pt,twoside]{article}
\usepackage{times}
\usepackage{amsmath,amssymb,amsthm}
\usepackage{enumerate}
\usepackage{cite}
\usepackage{mathrsfs,graphicx}

\allowdisplaybreaks

\newcommand{\no}{\nonumber}

\newcommand{\doubleint}{\int\!\!\!\!\!\int}

\newcommand{\R}{\mathbb R}

\newcommand{\p}{\partial}
\newcommand{\ve}{\varepsilon}
\newcommand{\f}{\frac}
\newcommand{\la}{\lambda}

\newcommand{\al}{\alpha}
\renewcommand{\t}{\tilde}

\newcommand{\g}{\gamma}

\newcommand{\dl}{\delta}
\newcommand{\ds}{\displaystyle}
\newcommand{\RN}[1]{\textup{\uppercase\expandafter{\romannumeral#1}}}
\newcommand{\mcR}{\mathcal{R}}

\def\defeq{\stackrel{\rm def}{=}}
\def\opdiv{\operatorname{div}}
\def\opcurl{\operatorname{curl}}
\def\supp{\operatorname{supp}}
\def\ls{\lesssim}
\def\gt{\gtrsim}

\theoremstyle{plain}
\newtheorem{theorem}{Theorem}[section]

\newtheorem{lemma}[theorem]{Lemma}

\theoremstyle{definition}

\theoremstyle{remark}
\newtheorem{remark}{Remark}[section]

\newtheorem*{acknowledgement}{Acknowledgement}

\numberwithin{equation}{section}

\title{On the global existence and blowup of smooth solutions to
the multi-dimensional compressible Euler equations with time-depending damping}

\author{Fei Hou$^{1, *}$ \qquad Huicheng
  Yin$^{2, }$\footnote{Fei Hou (\texttt{houfeimath@gmail.com}) and
    Huicheng Yin (\texttt{huicheng$@$nju.edu.cn}) were supported by
    the NSFC (No.~11571177) and the Priority Academic Program
    Development of Jiangsu Higher Education Institutions.}\\
    [12pt] {\small 1. Department of Mathematics and IMS,
  Nanjing University, Nanjing 210093, China}\\
  {\small 2. School of Mathematical Sciences, Nanjing
  Normal University, Nanjing 210023, China}}

\begin{document}
\date{}
\maketitle
\thispagestyle{empty}

\begin{abstract}
In this paper, we are concerned with the global existence and blowup of
smooth solutions to the multi-dimensional compressible Euler equations
with time-depending damping
\begin{equation*}
\left\{ \enspace
\begin{aligned}
  &\p_t\rho+\opdiv(\rho u)=0,\\
  &\p_t(\rho u)+\opdiv\left(\rho u\otimes u+p\,\RN{1}_d\right)=-\al(t)\rho u,\\
  &\rho(0,x)=\bar \rho+\ve\rho_0(x),\quad u(0,x)=\ve u_0(x),
\end{aligned}
\right.
\end{equation*}
where $x=(x_1, \cdots, x_d)\in\Bbb R^d$ $(d=2,3)$,
the frictional coefficient is $\al(t)=\frac{\mu}{(1+t)^\la}$ with $\la\ge0$ and
$\mu>0$, $\bar\rho>0$ is a constant, $\rho_0,u_0 \in C_0^\infty(\R^d)$,
$(\rho_0,u_0)\not\equiv 0$, $\rho(0,x)>0$, and $\ve>0$ is sufficiently small.
One can totally divide the range of $\la\ge0$ and $\mu>0$ into the following four cases:

Case 1: $0\le\la<1$,~$\mu>0$ for $d=2,3$;

Case 2: $\la=1$, $\mu>3-d$ for $d=2,3$;

Case 3: $\la=1$,~$\mu\le 3-d$ for $d=2$;

Case 4: $\la>1$,~$\mu>0$ for $d=2,3$.

\noindent We show that there exists a global $C^{\infty}-$smooth solution $(\rho, u)$ in Case 1, and Case 2 with $\opcurl u_0\equiv 0$,
while in Case 3 and Case 4, in general, the solution $(\rho, u)$ blows up in finite time.
Therefore, $\la=1$ and $\mu=3-d$ appear to be the critical power and critical value, respectively,
for the global existence of small amplitude smooth solution $(\rho, u)$ in $d-$dimensional compressible Euler equations
with time-depending damping.


\end{abstract}

\section{Introduction}
In this paper, we are concerned with the global existence and blowup of $C^{\infty}-$smooth solution
$(\rho, u)$ to the multi-dimensional compressible Euler
equations with time-depending damping
\begin{equation}\label{euler-eqn}
\left\{ \enspace
\begin{aligned}
&\p_t\rho+\opdiv(\rho u)=0,\\
&\p_t(\rho u)+\opdiv(\rho u\otimes u+p\,\RN{1}_d)=-\al(t)\rho u,\\
&\rho(0,x)=\bar \rho+\ve\rho_0(x),\quad u(0,x)=\ve u_0(x),
\end{aligned}
\right.
\end{equation}
where $x=(x_1,\cdots,x_d)\in\R^d$, $d=2,3$,
$\rho$, $u=(u_1,\cdots,u_d)$, and $p$ stand for the density, velocity and pressure,
respectively, $\RN{1}_d$ is the $d\times d$ identity matrix, the frictional coefficient
is $\al(t)=\frac{\mu}{(1+t)^\la}$ with $\la\ge0$ and $\mu>0$, and $u_0=(u_{1,0},\cdots,u_{d,0})$.
The state equation of the gases is described by $p(\rho)=A\rho^{\g}$, where $A>0$ and $\g>1$ are constants.
In addition, $\bar\rho>0$ is a constant, $\rho_0,u_0\in C_0^\infty(\R^d)$,
$\supp\rho_0,\supp u_0 \subseteq \{x\colon|x|\le M\}$, $(\rho_0, u_0)\not\equiv 0$,
$\rho(0,x)>0$, and $\ve>0$ is sufficiently small. For the physical background of \eqref{euler-eqn},
it can be found in \cite{Da} and the references therein.

For $\mu=0$ in $\al(t)$, \eqref{euler-eqn} is the standard compressible Euler equation.
It is well known that smooth solution $(\rho, u)$ of \eqref{euler-eqn} will generally blow up in finite time.
For examples, for a special class of initial data $(\rho(0,x), u(0,x))$,
Sideris~\cite{Sideris85} has proved that the smooth solution $(\rho, u)$ of \eqref{euler-eqn} in three space dimensions can
develop singularities in finite time, and Rammaha in \cite{Rammaha89} has proved a blowup result in two space dimensions.
For more extensive literature on the blowup results and the blowup mechanism for $(\rho, u)$,
see \cite{Alinhac93, Alinhac99a, Alinhac99b,Chr07,CM14,CL15,DWY16,Sideris97,Speck14,Yin04} and the references therein.

For $\la=0$ in $\al(t)$, it has been shown that \eqref{euler-eqn} admits a global
smooth solution $(\rho, u)$, moreover, the long-term behavior of the solution $(\rho,u)$
has been established, see \cite{HL92,HS96,KY04,Nish,PZ09,STW03,TW12,WY01,WY07}.
In particular, in \cite{STW03}, the authors showed that the vorticity of velocity $u$ decays to zero exponentially in time $t$.

For $\mu>0$ and $\la>0$ in $\al(t)$, one naturally asks:
does the smooth solution of \eqref{euler-eqn} blow up in finite time or does it exist globally?
For the case of  $\opcurl u_0\equiv0$, in \cite{HWY15}, we have studied this problem in three space dimensions
and proved that for $0\le\la\le1$ and $\mu>0$ there exists a global smooth solution $(\rho, u)$ of \eqref{euler-eqn}
and while for $\la>1$, in general, the solution will blow up in finite time.
In this paper, we will remove the assumption $\opcurl u_0\equiv0$ in \cite{HWY15}
and systematically study this problem both in two and three space dimensions.

Obviously, one can divide $\la\ge0$,~$\mu>0$ into four cases:

{\bf Case 1}: $0\le\la<1$,~$\mu>0$ for $d=2,3$;

{\bf Case 2}: $\la=1$,~$\mu>3-d$, for $d=2,3$;

{\bf Case 3}: $\la=1$,~$\mu\le 3-d$ for $d=2$;

{\bf Case 4}:  $\la>1$,~$\mu>0$ for $d=2,3$.

\vskip 0.1 true cm
At first, we state the global existence results in this paper.

\vskip 0.1 true cm

\begin{theorem}[\bf Global existence for Case 1]\label{thm1}
If $0\le\la<1$ and $\mu>0$, then for small $\ve>0$, \eqref{euler-eqn} admits a global $C^\infty-$ smooth solution
$(\rho, u)$ which fulfills $\rho>0$ and which is uniformly bounded for $t\ge0$ together with all its derivatives.
In addition, the vorticity $\opcurl u$ and its derivatives decay to zero in the rate $e^{-\frac{\mu}{3(1-\la)}[(1+t)^{1-\la}-1]}$,
where $\opcurl u=\p_1u_2-\p_2u_1$ for $d=2$, and $\opcurl u=(\p_2u_3-\p_3u_2, \p_3u_1-\p_1u_3, \p_1u_2-\p_2u_1)^T$ for $d=3$.
\end{theorem}

\vskip 0.1 true cm

\begin{theorem}[\bf Global existence for Case 2 with $\opcurl u_0\equiv0$]\label{thm2}
If $\la=1$, $\mu>3-d$ and $\opcurl u_0\equiv 0$, then for small $\ve>0$, \eqref{euler-eqn} admits a global $C^\infty-$ smooth solution
$(\rho, u)$ which fulfills $\rho>0$ and which is uniformly bounded for $t\ge0$ together with all its derivatives.
\end{theorem}

Next we concentrate on Case 3 and Case 4. As in \cite{Rammaha89}, we introduce the two functions
\begin{align*}
  q_0(l)&\defeq \int_{x_1>l}(x_1-l)^2\left(\rho(0,x)-\bar\rho\right)dx, \\
  q_1(l)&\defeq 2\int_{x_1>l}(x_1-l)(\rho u_1)(0,x)\,dx.
\end{align*}
Before stating our blowup result for problem \eqref{euler-eqn}, we require to introduce a special hypergeometric function $\Psi(a,b,c;z)$,
where the constants $a$ and $b$ satisfy $a+b=1$ and
\begin{equation*}
ab=\left\{
\begin{aligned}
  &\frac{\mu\la}{2}, && \la>1, \\
  &\frac\mu2(1-\frac\mu2), && \la=1,
\end{aligned}
\right.
\end{equation*}
$c\in\Bbb R^+$, the variable $z\in\Bbb R$, and
\begin{equation*}
\Psi(a,b,c;z)=\ds\sum_{n=0}^{+\infty}\f{(a)_n(b)_n}{n!(c)_n}z^n
\end{equation*}
with $(a)_n=a(a+1)\cdot\cdot\cdot(a+n-1)$ and $(a)_0=1$.
It is known from \cite{EMOT} that $\Psi(a,b,c;z)$ is an analytic function of $z$ for $z\in(-1, 1)$ and $\Psi(a,b,c;0)=\Psi(a+1,b+1,c;0)=1$.
In addition, there exists a small constant $\dl_0\in(0,1)$ depending on $\mu$ and $\la$ such that for $-\frac{\dl_0}{2}\le z\le 0$,
\begin{equation}\label{Psi-bound}
\frac12 \le \Psi(a,b,1;z), \Psi(a+1,b+1,2;z) \le \frac32.
\end{equation}

\begin{theorem}[\bf Blowup for Case 3 and Case 4]\label{thm3}
Suppose $\supp\rho_0,\supp u_0\subseteq \{x\colon|x|\le M\}$ and let
\begin{align}
q_0(l)&>0, \label{q0-positive}\\
q_1(l)&\ge0 \label{q1-positive}
\end{align}
hold for all $l\in (\t M, M)$, where $\t M$ is some fixed constant satisfying $0\le \t M<M$.
Moreover, we assume that there exist two constants $M_0$ and $\Lambda$ with $\max\{\t M, M-\dl_0\}\le M_0<M$ and $\Lambda\ge 3ab$ such that
\begin{equation}\label{+condition}
    q_1(l) \ge \Lambda q_0(l)
\end{equation}
holds for all $l\in (M_0, M)$.
If $\la=1$, $\mu\le1$ for $d=2$ or $\la>1$, $\mu>0$ for $d=2,3$, then there exists an $\ve_0>0$ such that,
for $0<\ve\le\ve_0$, the lifespan $T_\ve$ of the smooth solution $(\rho, u)$ of \eqref{euler-eqn} is finite.
\end{theorem}

\begin{remark} Our results in Theorem~\ref{thm1}-\ref{thm3} are strongly motivated by
considering the 1-D Burgers equation with time-depending damping term
\begin{equation}\label{burgers-eqn}
\left\{ \enspace
\begin{aligned}
&\p_t v+v\p_x v=-\,\ds\frac{\mu}{(1+t)^\la}\,v,\qquad (t,x)\in\R_+\times\R,\\
&v(0,x)=\ve v_0(x),
\end{aligned}
\right.
\end{equation}
where $\la\ge0$ and $\mu>0$ are constants, $v_0\in C_0^{\infty}(\R)$, $v_0\not\equiv 0$,
and $\ve>0$ is sufficiently small. One may directly obtain that by the method of characteristics
\begin{equation*}
\left\{ \enspace
\begin{aligned}
  &T_{\ve}=\infty, & \text{if $0\le\la<1$ or $\la=1$, $\mu>1$,}\\
  &T_{\ve}<\infty, & \text{if $\la>1$ or $\la=1$, $0<\mu\le 1$,}
\end{aligned}
\right.
\end{equation*}
where $T_{\ve}$ is the lifespan of the smooth solution $v$ of \eqref{burgers-eqn}.
Especially in the case of $0\le\la<1$, $v$ exponentially decays to zero with respect to the time $t$.
This means that $\la=1$ and $\mu=1$ appear to be the critical power and critical value respectively,
for the global existence of smooth solution $v$ of \eqref{burgers-eqn}.
\end{remark}

\begin{remark}
For the three dimensional problem \eqref{euler-eqn} and the case $\la=0$ in $\al(t)$,
the authors in \cite{STW03} proved that the fluid vorticity decays to zero  exponentially in time,
while the solution $(\rho, u)$ does not decay exponentially.
In {\bf Case 1} of $0\le\la<1$ and $\mu>0$, we have precisely proved that the vorticity $\opcurl u$
decays to zero in the rate $e^{-\frac{\mu}{3(1-\la)}[(1+t)^{1-\la}-1]}$ in Theorem~\ref{thm1}.
\end{remark}

\begin{remark}
In Theorem~\ref{thm2}, we pose the assumption of $\opcurl u_0\equiv 0$ for {\bf Case 2}.
If not, it seems difficult for us to obtain the uniform control on the vorticity $\opcurl u$
by our method. Namely, so far we do not know whether the assumption of $\opcurl u_0\equiv 0$
can be removed in order to obtain the global existence of $(\rho, u)$ in {\bf Case 2}.
\end{remark}

\begin{remark}
It is not hard to find a large number of initial data $(\rho,u)(0,x)$
such that \eqref{q0-positive}-\eqref{+condition} are satisfied.
For instance, choosing $\rho_0(x)>0$ and $u_{1,0}(x)=x_1\rho_0(x)\Lambda/\bar\rho$,
then we get \eqref{q0-positive}-\eqref{+condition}.
\end{remark}

\begin{remark}\label{rem1.6} In \cite{Sideris85} and \cite{{Rammaha89}}, the authors have
shown the formation of singularities in multi-dimensional compressible Euler equations
(corresponding $\mu=0$ in \eqref{euler-eqn}) under the assumptions of \eqref{q0-positive}-\eqref{q1-positive}.
However, in order to prove the blowup result of smooth solution $(\rho, u)$ to problem \eqref{euler-eqn} and
overcome the difficulty arisen by the time-depending frictional coefficient $\f{\mu}{(1+t)^\la}$ with $\mu>0$ and $\la\ge1$,
we pose an extra assumption \eqref{+condition} except \eqref{q0-positive}-\eqref{q1-positive},
which leads to the non-negativity lower bound of $P(t,l)$ in \eqref{5.8}
so that two ordinary differential blowup inequalities \eqref{5.19}-\eqref{5.20} can be
established. One can see more details in $\S 5$.
\end{remark}

\begin{remark}
If the damping term $\al(t)\rho u$ in \eqref{euler-eqn} is replaced by $(\al_1(t)\rho u_1,\cdots,\al_d(t)\rho u_d)^T$
with $\al_i(t)=\frac{\mu_i}{(1+t)^{\la_i}}$ ($i=1,\cdots,d$), and there exists some $i_0$ $(1\le i_0\le d)$
such that $\la_{i_0}$ and $\mu_{i_0}$ satisfy {\bf Case 3} or {\bf Case 4}.
In this case, one can define the new quantities
\begin{align*}
  q_0(l)&= \int_{x_{i_0}>l}(x_{i_0}-l)^2\left(\rho(0,x)-\bar\rho\right)dx, \\
  q_1(l)&= 2\int_{x_{i_0}>l}(x_{i_0}-l)(\rho u_2)(0,x)\,dx
\end{align*}
and
\begin{equation*}
    P(t,l)=\int_{x_{i_0}>l}(x_{i_0}-l)^2\left(\rho(t,x)-\bar\rho\right)dx
\end{equation*}
instead of the ones in \eqref{q0-positive}-\eqref{q1-positive} and \eqref{5.1}, respectively,
we then obtain an analogous result in Theorem~\ref{thm3} by applying the same procedure in $\S 5$.
\end{remark}

Let us indicate the proofs of Theorems~\ref{thm1}-\ref{thm3}. Without loss of generality, from now on we
assume that $\bar c=c(\bar\rho)=1$, where $c(\rho)=\sqrt{P'(\rho)}$ is the sound speed.
At first, we reformulate problem \eqref{euler-eqn}. Set
\begin{equation}\label{theta-def}
    \theta \defeq \frac1{\g-1}(A\g\rho^{\g-1}-1)=\frac1{\g-1}(c^2(\rho)-1).
\end{equation}
Then problem \eqref{euler-eqn} can be rewritten as
\begin{equation}\label{euler-reform}
    \left\{ \enspace
\begin{aligned}
&\p_t\theta+u\cdot\nabla\theta+(1+(\g-1)\theta)\opdiv u=0, \\
&\p_tu+\frac\mu{(1+t)^{\la}}u+u\cdot\nabla u+\nabla\theta=0, \\
&\theta(0,x)=\frac{1}{\g-1}[(1+\frac{\ve\rho_0(x)}{\bar\rho})^{\g-1}-1]
    \defeq \ve\theta_0(x)+\ve^2g(x,\ve), \\
& u(0,x)=\ve u_0(x),
\end{aligned}
\right.
\end{equation}
where $\nabla=(\p_1,\cdots,\p_d)=(\p_{x_1},\cdots,\p_{x_d})$, $\theta_0(x)=\frac{\rho_0(x)}{\bar\rho}$ and
$g(x,\ve)=(\g-2)\frac{\rho_0^2(x)}{\bar\rho^2}\int_0^1 (1+\frac{\sigma\ve\rho_0(x)}{\bar\rho})^{\g-3}(1-\sigma)\,d\sigma$.
Note that $g(x,\ve)$ is smooth in $(x,\ve)$ and has compact support in $x$.

To prove Theorem~\ref{thm1}, we introduce such a time-weighted energy
\begin{equation}\label{energy1}
    \mathcal{E}_k[\Phi](t) \defeq (1+t)^\la\sum_{1\le|\al|+j\le k}
        \|\p_t^j\nabla^\al\Phi(t,\cdot)\|+\|\Phi(t,\cdot)\|,
\end{equation}
where $k$ is a fixed positive number, and $\|\cdot\|$ stands for the $L_x^2$ norm on $\R^d$, i.e.,
\[
  \|\Phi(t,\cdot)\| \defeq \|\Phi(t,x)\|_{L_x^2(\R^d)}=
    \left(\int_{\R^d}|\Phi(t,x)|^2dx\right)^\frac12.
\]
Denote by
\begin{equation}\label{energy+}
\mathcal{E}_k[\Phi_1,\Phi_2](t) \defeq \mathcal{E}_k[\Phi_1](t)+\mathcal{E}_k[\Phi_2](t).
\end{equation}
For  $0\le\la<1$ and $\mu>0$, one can choose a constant $t_0$ such that
\begin{equation}\label{critical-time}
    (1+t_0)^{1-\la}=\max\,\{\frac2\mu, 1\},
\end{equation}
so that problem \eqref{euler-reform} has a local solution $(\theta, u)\in C^{\infty}([0, t_0]\times\R^3)$
by the smallness of $\ve>0$ (see the local existence result for the multidimensional hyperbolic systems in \cite{Ma}).
Making use of the vorticity $\opcurl u$ and the conditions of $0\le\la<1$ and $\mu>0$ in {\bf Case 1},
and simultaneously taking the delicate analysis on the system \eqref{euler-reform},
the uniform time-weighted energy estimates for $\mathcal{E}_4[\theta, u](t)$ are obtained.
This, together with the continuity argument, yields the proof of Theorem~\ref{thm1}.

Since we have proved Theorem~\ref{thm2} in \cite{HWY15} for the {\bf Case 2} with $\opcurl u_0\equiv 0$ in three space dimensions,
we only require to focus on the proof of Theorem~\ref{thm2} in two space dimensions.
For this purpose, we define another energy
\begin{equation}\label{energy2}
    E_k[\Phi](t) \defeq (1+t)^\frac12\sum_{0\le|\al|\le k-1}
        \|\p Z^\al\Phi(t,\cdot)\|+(1+t)^{-\frac12}\|\Phi(t,\cdot)\|,
\end{equation}
where $\p=(\p_t, \p_{x_1}, \p_{x_2})$, $Z=(Z_0, Z_1, \dots, Z_6)=(\p, S, R, H)$ with the scaling field $S=t\p_t+x_1\p_1+x_2\p_2$,
the rotation field $R=x_1\p_2-x_2\p_1$, the Lorentz fields $H=(H_1, H_2)=(x_1\p_t+t\p_1, x_2\p_t+t\p_2)$
and $Z^\al=Z_0^{\al_0}Z_1^{\al_1}\cdots Z_6^{\al_6}$.
From \eqref{euler-reform} we may derive a damped wave equation of $\theta$ as follows
\begin{equation}\label{damped-wave}
    \p_t^2\theta+\frac\mu{1+t}\p_t\theta-\Delta\theta=Q(\theta,u),
\end{equation}
where the expression of $Q(\theta,u)$ will be given in \eqref{Q-def} below.
Thanks to $\opcurl u\equiv0$, we can get the estimates of velocity $u$
from the equations in \eqref{euler-reform} (see Lemma~\ref{lem-Zvelocity}).
By $\mu>1$ and a rather technical analysis on the damped wave equation \eqref{damped-wave},
we eventually show in $\S 4$ that $E_5[\theta,u](t) \le \frac12 \,K_3\ve$ (see \eqref{energy+} for the
definition of $E_5[\theta,u](t)$) holds when
$E_5[\theta,u](t) \le K_3\ve$ is assumed for some suitably large constant $K_3>0$ and small $\ve>0$.
Based on this and the continuity argument, the global existence of $(\theta, u)$ and then
Theorem~\ref{thm2} in two space dimensions are established for $\la=1$,~$\mu>1$ and $\opcurl u_0\equiv0$.

To prove the blowup result in Theorem~\ref{thm3}, as in \cite{Rammaha89,Sideris85},
we shall derive some blowup-type second-order ordinary differential inequalities in $\S 5$.
From this and assumptions \eqref{q0-positive}-\eqref{+condition}, an upper bound of the lifespan $T_{\ve}$
is derived by making use of $\la=1$, $\mu\le 3-d$ or $\la>1$, and then the proof of Theorem~\ref{thm3} is completed.

Here we point out that in \cite{HWY15}, for the 3-d {\bf irrotational} compressible Euler equations,
it has been shown that for $0\le\la\le1$, there exists a global $C^{\infty}-$smooth small amplitude solution $(\rho, u)$,
while for $\la>1$, the smooth solution $(\rho, u)$ generally blows up in finite time.
This means that we have extended the global existence and blowup results in \cite{HWY15} for the
3-D irrotational flows to the 2-D and 3-D full Euler systems.

In the whole paper, we shall use the following convention:

\begin{itemize}
\item $C$ will denote a generic positive constant which is independent
      of $t$ and $\ve$.

\item $A\ls B$ or $B\gt A$ means $A\le CB$.

\item $r=|x|=\sqrt{x_1^2+\cdots+x_d^2}$,~~$\sigma_-(t,x)\defeq
      \sqrt{1+(r-t)^2}$.

\item $\|\Phi(t,\cdot)\| \defeq \|\Phi(t,x)\|_{L_x^2(\R^d)}$,~~
      $|\Phi(t,\cdot)|_\infty \defeq |\Phi(t,x)|_{L_x^\infty}=
      \ds\sup_{x\in\R^d}|\Phi(t,x)|$.

\item $Z$ denotes one of the Klainerman vector fields $\{\p, S, R, H\}$
      on $\R_+\times\R^2$, where $\p=(\p_t, \p_{x_1}, \p_{x_2})$,
      $S=t\p_t+x_1\p_1+x_2\p_2$, $R=x_1\p_2-x_2\p_1$ and $H=(H_1, H_2)=
      (x_1\p_t+t\p_1, x_2\p_t+t\p_2)$.

\item For two vector fields $X$ and $Y$, $[X,Y] \defeq XY-YX$ denotes
      the Lie bracket.

\item Greek letters $\al,\beta,\cdots$ denote multiple indices, i.e.,
      $\al=(\al_0,\cdots,\al_m)$,  and $|\al|=\al_0+\cdots+\al_m$ denotes
      its length, where $\al_i$ is some non-negative integer for all $i=0,\cdots,m$.

\item For two multiple indices $\al$ and $\beta$, $\beta\le\al$ means
      $\beta_i\le\al_i$ for all $i=0,\cdots,m$ while $\beta<\al$ means
      $\beta\le\al$ and $\beta_i<\al_i$ for some $i$.

\item For the differential operator $O=(O_0,\cdots,O_m)$, for example, $O=(\p_t,
      \p_{x_1},\cdots,\p_{x_d})$ in $\S 3$  and $O=
      (\p_t, \p_{x_1}, \p_{x_2}, S, R, H)$ in $\S 4$,
      denote $O^\al \defeq O_0^{\al_0}\cdots O_m^{\al_m}$, $O^{\le\al}
      \defeq \ds\sum_{0\le\beta\le\al}O^\beta$, $O^{<\al} \defeq \ds\sum_
      {0\le\beta<\al}O^\beta$ and $O^{\le k} \defeq \ds\sum_{0\le|\al|\le k}
      O^\al$ with $k$ is an integer.

\item Leibniz's rule: $O^\al(\Phi\Psi)=\ds\sum_{0\le\beta\le\al}C_{\al,\beta}
      O^\beta\Phi O^{\al-\beta}\Psi$ will be abbreviated as \\
      $O^\al(\Phi\Psi)=\ds\sum_{0\le\beta\le\al}O^\beta\Phi O^{\al-\beta}\Psi$.

\item $\Xi$ is the solution of $\ds \Xi'(t) = \frac{\mu}{(1+t)^\la}\,
      \Xi(t)$ with $\Xi(0)=1$, i.e.,
      \begin{equation}\label{Xi-def}
      \Xi(t)\defeq \begin{cases} e^{\frac{\mu}{1-\la}[(1+t)^{1-\la}-1]}, &
      \la\ge 0,\,\la\neq1,\\ (1+t)^\mu, & \la=1. \end{cases}
      \end{equation}

\item 
      $c(\bar\rho)=1$ will be assumed throughout (otherwise, introduce
      $X=x/c(\bar\rho)$ as new space coordinate if necessary).
\end{itemize}

\smallskip


\section{Some Preliminaries}\label{section2}

At first, we derive the scalar equation of $\theta$ in \eqref{euler-reform}.
It follows from the first equation in
\eqref{euler-reform} that
\begin{equation}\label{dtdivu1}
  \p_t\opdiv u=-\frac1{(1+(\g-1)\theta)} (\p_t^2\theta+u\cdot\nabla\p_t\theta
    +\p_tu\cdot\nabla\theta+(\g-1)\p_t\theta\opdiv u).
\end{equation}
Taking divergence on the second equation in \eqref{euler-reform} yields
\begin{equation}\label{dtdivu2}
    \opdiv\p_t u+\frac\mu{(1+t)^\la}\opdiv u+\Delta\theta+
        u\cdot\nabla\opdiv u+\sum_{i,j=1}^d \p_iu_j\p_ju_i=0,
\end{equation}
where $\Delta=\p_1^2+\cdots+\p_d^2$.
Substituting \eqref{dtdivu1} into \eqref{dtdivu2} yields the damped wave equation of $\theta$
\begin{equation}\label{theta-eqn}
    \p_t^2\theta+\frac\mu{(1+t)^\la}\p_t\theta-\Delta\theta=Q(\theta,u),
\end{equation}
where
\begin{align}
    Q(\theta,u)   &\defeq Q_1(\theta,u)+Q_2(\theta,u), \label{Q-def}\\
    Q_1(\theta,u) &\defeq (\g-1)\theta\Delta\theta-\frac\mu{(1+t)^\la}
        u\cdot\nabla\theta-2u\cdot\nabla\p_t\theta-\sum_{i,j=1}^d u_iu_j\p_{ij}^2\theta \label{Q1-def},\\
    Q_2(\theta,u) &\defeq -\sum_{i,j=1}^d u_i\p_iu_j\p_j\theta-\p_tu\cdot\nabla\theta
        +(1+(\g-1)\theta) (\sum_{i,j=1}^d \p_iu_j\p_ju_i+(\g-1)|\opdiv u|^2). \label{Q2-def}
\end{align}
Let
\begin{equation}\label{vorticity-def}
w\defeq \opcurl u=
\begin{cases}
    \p_1u_2-\p_2u_1, & d=2, \\
    (\p_2u_3-\p_3u_2, \p_3u_1-\p_1u_3, \p_1u_2-\p_2u_1)^T, & d=3.
\end{cases}
\end{equation}
Then the second equation in \eqref{euler-reform} implies that for $d=2$
\begin{equation}\label{2dvorticity-eqn}
    \p_tw+\frac\mu{(1+t)^\la}w+u\cdot\nabla w+w\opdiv u=0
\end{equation}
and for $d=3$
\begin{equation}\label{3dvorticity-eqn}
    \p_tw+\frac\mu{(1+t)^\la}w+u\cdot\nabla w+w\opdiv u=w\cdot\nabla u.
\end{equation}
To prove Theorem~\ref{thm1}-\ref{thm2}, we require to introduce the following lemma, which is easily shown.
\begin{lemma}\label{lem-divcurl}
Let $U(x)=(U_1(x),\cdots,U_d(x))$ be a vector-valued function with compact support on $\R^d$ ($d=2,3$), then there holds that
\begin{equation}\label{divcurl-estimate}
    \|\nabla U\|\le \|\opcurl U\|+\|\opdiv U\|.
\end{equation}
\end{lemma}

The following Sobolev type inequality can be found in \cite{Kl87}.
\begin{lemma}\label{lem-Klainerman-ineq}
Let $\Phi(t,x)$ be a function on $\R^{1+2}$, then there exists a constant
$C$ such that
\begin{equation}\label{Klainerman-ineq}
    (1+t+r)\,\sigma_-(t,x)\,|\Phi(t,x)| \le C\sum_{|\al|\le2} \|Z^\al\Phi(t,\cdot)\|^2.
\end{equation}
\end{lemma}
In addition, we have
\begin{lemma}\label{lem-weight}
Let $\Phi(t,x)$ be a function on $\R^{1+2}$ and assume $\supp\Phi\subseteq\{(t,x)\colon |x|\le t+M\}$,
then there exists a constant $C>0$ such that for $\nu\in (-\infty,1)$
\begin{equation}\label{weight-pointwise}
    |\sigma_-^{\nu-1}(t,\cdot)\Phi(t,\cdot)|_\infty \le C|\sigma_-^{\,\nu}
        (t,\cdot)\nabla\Phi(t,\cdot)|_\infty
\end{equation}
and for $\ell\neq 1$
\begin{equation}\label{weight-L2}
    \|\sigma_-^{-\ell}(t,\cdot)\Phi(t,\cdot)\| \le
        C(t+M)^{(1-\ell)_+}\,\|\nabla\Phi(t,\cdot)\|,
\end{equation}
where $(1-\ell)_+=\max\,\{1-\ell, 0\}$ and $\ell\in\R$.
\end{lemma}
\begin{proof}
For the purpose of completeness, we prove \eqref{weight-pointwise}-\eqref{weight-L2} here.
In fact, for the proof of \eqref{weight-L2}, one can also see \cite[Lemma~2.2]{Alinhac93}.
By introducing the polar coordinate $(r,\phi)$ such that $x_1=r\cos\phi$ and $x_2=r\sin\phi$,
we then have
\begin{align}
    \Phi(t,x)&=\Phi(t,r\cos\phi,r\sin\phi)=-\int_r^{t+M} \frac{d}{d\xi}\,
        \Phi(t,\xi\cos\phi,\xi\sin\phi) \,d\xi \no\\
    &=-\int_r^{t+M} (\cos\phi\,\p_1\Phi(t,\xi\cos\phi,\xi\sin\phi)+
        \sin\phi\,\p_2\Phi(t,\xi\cos\phi,\xi\sin\phi) ) \,d\xi. \label{2.14}
\end{align}
Together with the mean value theorem, this yields
\begin{equation*}
    \Phi(t,x) \ls |\sigma_-^{\,\nu}(t,\cdot)\nabla\Phi(t,\cdot)|_\infty
        \int_r^{t+M} (1+|t-\xi|)^{-\nu} \,d\xi,
\end{equation*}
which immediately derives \eqref{weight-pointwise}.
On the other hand, applying Cauchy-Schwartz inequality to \eqref{2.14} derives
\begin{equation*}
    |\Phi(t,r\cos\phi,r\sin\phi)|^2 \le \left( \int_r^{t+M}
        |\nabla\Phi(t,\xi\cos\phi,\xi\sin\phi)|^2 \,\xi\,d\xi\right)
        \left( \int_r^{t+M} \frac1\xi \,d\xi\right),
\end{equation*}
which yields
\begin{align}
    \left\|\frac{\Phi(t,\cdot)}{(t+2M-r)^\ell}\right\|^2
    &= \int_0^{2\pi}\int_0^{t+M} \frac{|\Phi(t,r\cos\phi,r\sin\phi)|^2}
        {(t+2M-r)^{2\ell}} \,r\,drd\phi \no\\
    &\le \int_0^{2\pi}\int_0^{t+M} \frac{r\int_r^{t+M} \frac1\xi \,d\xi}
        {(t+2M-r)^{2\ell}} \,dr \int_r^{t+M} |\nabla\Phi(t,\xi\cos\phi,
        \xi\sin\phi)|^2 \,\xi\,d\xi d\phi \no\\
    &\ls \int_0^{t+M} \frac{r\log\frac{t+M}r}{(t+2M-r)^{2\ell}} \,dr
        \,\|\nabla\Phi(t,\cdot)\|^2. \label{2.15}
\end{align}
On the other hand, it follows from direct computation that
\begin{equation}\label{2.16}
    \int_0^\frac{t+M}{2}\, \frac{r\log\frac{t+M}r}{(t+2M-r)^{2\ell}} \,dr
        \ls \frac1{(t+M)^{2\ell}} \int_0^\frac{t+M}{2}\, r\log\frac{t+M}r \,dr
        \ls (t+M)^{2(1-\ell)}
\end{equation}
and
\begin{equation}\label{2.17}
    \int_\frac{t+M}{2}^{t+M}\, \frac{r\log\frac{t+M}r}{(t+2M-r)^{2\ell}}\,dr
        =\int_0^\frac{t+M}{2}\, \frac{(t+M-\xi)\log\frac{t+M}{t+M-\xi} }{(M+\xi)^{2\ell}} \,d\xi
        \le \int_0^\frac{t+M}{2} \,\frac\xi{(M+\xi)^{2\ell}} \,d\xi,
\end{equation}
where we have used the fact of $\frac{t+M}{t+M-\xi}=1+\frac\xi{t+M-\xi} \le e^\frac\xi{t+M-\xi}$
in the last inequality.

Substituting \eqref{2.16}-\eqref{2.17} into \eqref{2.15} and taking direct computation
yield \eqref{weight-L2}.
Thus, the proof of Lemma~\ref{lem-weight} is completed.
\end{proof}


\section{Proof of Theorem 1.1.}\label{section3}


Throughout this section, we will always assume that $\mathcal{E}_4[\theta,u](t) \le K_1\ve$ holds,
where the definition of $\mathcal{E}_4[\theta,u](t)$ has been given in \eqref{energy1} and \eqref{energy+}.
Together with the standard Sobolev embedding theorem, this yields
\begin{equation}\label{3.1}
    |(\theta,u)(t,\cdot)|_\infty+(1+t)^\la
        |\p\p^{\le1}(\theta,u)(t,\cdot)|_\infty \ls K_1\ve.
\end{equation}
To prove Theorem~\ref{thm1}, we now carry out the following parts.

\subsection{Estimates of velocity $u$ and vorticity $w$.}

The following lemma is an application of Lemma~\ref{lem-divcurl} and \eqref{3.1}.
\begin{lemma}\label{lem-velocity1}
Under assumption \eqref{3.1}, for all $t>0$, one has
\begin{equation}\label{3.2}
    \mathcal{E}_4[u](t)\ls \|u(t,\cdot)\|+(1+t)^\la
        \left(\|\p^{\le3} w(t,\cdot)\|+\|\p\p^{\le3}\theta(t,\cdot)\|\right),
\end{equation}
where the definition of $w$ has been given in \eqref{vorticity-def}.
\end{lemma}
\begin{proof}
By the equations in \eqref{euler-reform}, we see that
\begin{align}
    \opdiv u &= -\frac{\p_t\theta+u\cdot\nabla\theta}{1+(\g-1)\theta}, \label{3.3} \\
    \p_tu &= -\left(u\cdot\nabla u+\frac\mu{(1+t)^\la}u+\nabla\theta\right). \label{3.4}
\end{align}
Taking $U=\p^\al u$ with $|\al|\le3$ in \eqref{divcurl-estimate}, we then arrive at
\begin{align}
    \|\nabla\p^{\le3}& u(t,\cdot)\|
    \ls \|\p^{\le3} w(t,\cdot)\|+\|\p^{\le3}\opdiv u(t,\cdot)\| \no\\
    &\ls \|\p^{\le3} w(t,\cdot)\|+\|\p_t\p^{\le3}\theta(t,\cdot)\|+ K_1\ve \left( \|\nabla\p^{\le3}\theta(t,\cdot)\|+
        (1+t)^{-\la}\|\p^{\le3} u(t,\cdot)\|\right), \label{3.5}
\end{align}
where we have used \eqref{3.1} and \eqref{3.3} in the last inequality.
Taking the $L^2$ norm of $\p^\al$\eqref{3.4} yields
\begin{align}
    \|\p_t\p^\al u(t,\cdot)\|
    &\ls \|\nabla\p^\al\theta(t,\cdot)\|+(1+K_1\ve)(1+t)^{-\la}
        \|\p^{\le\al} u(t,\cdot)\| +K_1\ve \|\nabla\p^{\le\al} u(t,\cdot)\|. \label{3.6}
\end{align}
Rewrite $\p^\al=\p_t^k\p_x^\beta$ with $0\le k+|\beta|\le3$.
Summing up \eqref{3.6} from $k=0$ to $k=3$ yields
\begin{align}
    \|\p_t\p^{\le3} u(t,\cdot)\|
    &\ls \|\nabla\p^{\le3}\theta(t,\cdot)\|+(1+t)^{-\la}\|u(t,\cdot)\|
     +K_1\ve \|\nabla\p^{\le3} u(t,\cdot)\|. \label{3.7}
\end{align}
By the smallness of $\ve>0$, we immediately derive \eqref{3.3} from \eqref{3.5} and \eqref{3.7}.
This completes the proof of Lemma~\ref{lem-velocity1}.
\end{proof}

The following lemma shows the estimate of velocity $u$ itself.
\begin{lemma}\label{lem-velocity2}
Let $\mu>0$. Under assumption \eqref{3.1}, for all $t>0$, it holds that
\begin{equation}\label{3.8}
    \frac{d}{dt}\|(\theta,u)(t,\cdot)\|^2+\frac\mu{(1+t)^\la}\|u(t,\cdot)\|^2
    \ls (1+t)^\la|\theta(t,\cdot)|_\infty\|\nabla\theta(t,\cdot)\|^2.
\end{equation}
\end{lemma}
\begin{proof}
Multiplying the second equation in \eqref{euler-reform} by $u$ derives
\begin{equation}\label{3.9}
  \frac12\,\p_t|u|^2+\frac\mu{(1+t)^\la}|u|^2+u\cdot\nabla\theta=
    -\frac12\,u\cdot\nabla|u|^2.
\end{equation}
From the first equation in \eqref{euler-reform}, we see that
\begin{align}
    u\cdot\nabla\theta &= \opdiv(\theta u)-\theta\opdiv u \no\\
    &= \opdiv(\theta u)+\theta(\p_t\theta+u\cdot\nabla\theta+
        (\g-1)\theta\opdiv u) \no\\
    &= \opdiv(\theta u+(\g-1)\theta^2 u)+\frac12\,\p_t|\theta|^2+
        (3-2\g)\theta\,u\cdot\nabla\theta.  \label{3.10}
\end{align}
Substituting \eqref{3.10} into \eqref{3.9} and integrating it over $\R^d$ yield
\begin{align}
    &\quad \frac{d}{dt}\|(\theta,u)(t,\cdot)\|^2+\frac{2\mu}{(1+t)^\la}
        \|u(t,\cdot)\|^2 \no\\
    &\ls |\theta(t,\cdot)|_\infty\|u(t,\cdot)\|\,\|\nabla\theta(t,\cdot)\|
        +|\nabla u(t,\cdot)|_\infty\|u(t,\cdot)\|^2. \label{3.11}
\end{align}
Substituting \eqref{3.1} into \eqref{3.11} and applying $\mu>0$ and the smallness of $\ve$, we derive \eqref{3.8}.
This completes the proof of Lemma~\ref{lem-velocity2}.
\end{proof}

Next lemma shows the estimates of vorticity $w$ and its derivatives.
\begin{lemma}\label{lem-vorticity}
Let $\mu>0$. Under assumption \eqref{3.1}, for all $t>0$, it holds that
\begin{equation}\label{3.12}
    \frac{d}{dt}\|\p^{\le3} w(t,\cdot)\|^2+\frac\mu{(1+t)^\la}
        \|\p^{\le3} w(t,\cdot)\|^2
    \ls |\p^{\le1} w(t,\cdot)|_\infty\, \|\p^{\le3} w(t,\cdot)\|\,
        \|\p\p^{\le3} u(t,\cdot)\|.
\end{equation}
\end{lemma}
\begin{proof}
It follows from vorticity equation \eqref{2dvorticity-eqn}-\eqref{3dvorticity-eqn} that
for $|\al|\le3$,
\begin{align}
    & \p_t\p^\al w+\frac{\mu}{(1+t)^\la}\p^\al w+u\cdot\nabla\p^\al w \no\\
    &= -\sum_{0<\beta\le\al}\left[\p^\beta\left(\frac\mu{(1+t)^\la}\right)
        \p^{\al-\beta} w+\p^\beta u\cdot\nabla\p^{\al-\beta} w\right]
        -\p^\al(w\opdiv u-w\cdot\nabla u),\label{3.13}
\end{align}
here we point out that the last term $w\cdot\nabla u$ in \eqref{3.13} does not appear when $d=2$ .
Multiplying \eqref{3.13} by $\p^\al w$ and integrating it over $\R^d$ yield
\begin{align}
    &\quad \frac{d}{dt}\|\p^\al w(t,\cdot)\|^2+\frac{2\mu}{(1+t)^\la}
        \|\p^\al w(t,\cdot)\|^2 \no\\
    &\ls \frac1{(1+t)^{1+\la}}\|\p^\al w(t,\cdot)\|\,\|\p^{<\al}w(t,\cdot)\|+
        |\p\p^{\le1} u(t,\cdot)|_\infty \|\p^{\le\al}w(t,\cdot)\|^2 \no\\
    &\quad +|\p^{\le1}w(t,\cdot)|_\infty \|\p\p^{\le3} u(t,\cdot)\|\,
        \|\p^{\le\al}w(t,\cdot)\|. \label{3.14}
\end{align}
Note that when $\al=0$, the firs term $\|\p^\al w(t,\cdot)\|\,\|\p^{<\al}w(t,\cdot)\|$ in the right hand side of \eqref{3.14} does not appear.
Summing up \eqref{3.14} from $|\al|=0$ to $|\al|=3$ and applying \eqref{3.1}, $\mu>0$ and the smallness of $\ve$, we then obtain \eqref{3.12}.
This completes the proof of Lemma~\ref{lem-vorticity}.
\end{proof}
\begin{remark}\label{rmk3.1}
The proof of Lemma~\ref{lem-velocity2} and Lemma~\ref{lem-vorticity} only depends on \eqref{3.1}, $\mu>0$ and the smallness of $\ve$.
\end{remark}

\subsection{Estimates of $\theta$ and its derivatives.}

The next lemma shows the global estimates of $\theta$ and its derivatives for $t>t_0$,
where $t_0$ is defined in \eqref{critical-time}.
\begin{lemma}\label{lem-thetaglobal}
Let $0\le\la<1$, $\mu>0$. Under assumption \eqref{3.1}, for all $t>t_0$, it holds that
\begin{align}
    &\quad \mathcal{E}_4^2[\theta](t)
        +\int_{t_0}^t(1+s)^\la\|\p\p^{\le3}\theta(s,\cdot)\|^2\,ds \no\\
    & \ls \mathcal{E}_4^2[\theta](t_0)+K_1\ve\int_{t_0}^t\left((1+s)^{-\la}
        \|u(s,\cdot)\|^2+(1+s)^\la\|\p^{\le3} w(s,\cdot)\|^2\right)\,ds. \label{3.15}
\end{align}
\end{lemma}
\begin{proof}
Acting $\p^\al$ with $|\al|\le3$ on both sides of equation \eqref{theta-eqn} shows
\begin{equation*}
    \p_t^2\p^\al\theta+\frac\mu{(1+t)^\la}\p_t\p^\al\theta-\Delta\p^\al\theta
    =\p^\al Q(\theta,u)+\sum_{\beta<\al} \p^{\al-\beta}\left(\frac\mu{(1+t)^\la}
        \right) \p_t\p^\beta\theta.
\end{equation*}
Multiplying this by $2(1+t)^{2\la}\p_t\p^\al\theta+\mu(1+t)^\la\p^\al\theta$ derives
\begin{align}
    &\p_t\left[(1+t)^{2\la}|\p\p^\al\theta|^2+\mu(1+t)^\la\p^\al\theta
        \p_t\p^\al\theta+\frac{\mu^2}2|\p^\al\theta|^2-\frac{\mu\la}{2}(1+t)^{\la-1}|\p^\al\theta|^2\right] \no\\
    & +\left[\mu(1+t)^\la-2\la(1+t)^{2\la-1}\right]|\p\p^\al\theta|^2-
        \opdiv\left[\nabla\p^\al\theta\left(2(1+t)^{2\la}\p_t\p^\al\theta
        +\mu(1+t)^\la\p^\al\theta\right) \right]\no\\
    &= \left(2(1+t)^{2\la}\p_t\p^\al\theta+\mu(1+t)^\la\p^\al\theta\right)
        \left(\p^\al Q(\theta,u)+\sum_{\beta<\al} \p^{\al-\beta}\left(
        \frac\mu{(1+t)^\la}\right) \p_t\p^\beta\theta\right) \no\\
    &\qquad +\frac{\mu\la(1-\la)}{2}(1+t)^{\la-2}|\p^\al\theta|^2. \label{3.16}
\end{align}
Thanks to $0\le\la<1$, $\mu>0$ and the choice of $t_0$ (see \eqref{critical-time}),
for all $t>t_0$ one easily knows that for the term in the first square bracket of the second line in \eqref{3.16},
\begin{equation}\label{3.17}
    \mu(1+t)^\la-2\la(1+t)^{2\la-1} \ge \mu(1-\la)(1+t)^\la.
\end{equation}
Furthermore, one gets that for the term in the square bracket of the first line in \eqref{3.16},
\begin{align*}
    (1+t)^{2\la}|\p\p^\al\theta|^2+\mu(1+t)^\la\p^\al\theta
        \p_t\p^\al\theta+\frac{\mu^2}{2}|\p^\al\theta|^2-\frac{\mu\la}{2}(1+t)^{\la-1}|\p^\al\theta|^2 \\
    =(1+t)^{2\la}\left(\frac{1-\la}{3-\la}|\p_t\p^\al\theta|^2+|\nabla\p^\al\theta|^2\right)
        +\frac{\mu^2(1-\la)}{8}|\p^\al\theta|^2 \\
    +\left((1+t)^\la\sqrt\frac{2}{3-\la}\p_t\p^\al\theta+\frac\mu2\sqrt\frac{3-\la}{2}\p^\al\theta\right)^2
        +\frac{\mu\la}{4}(\mu-2(1+t)^{\la-1})|\p^\al\theta|^2,
\end{align*}
which is equivalent to $(1+t)^{2\la}|\p\p^\al\theta|^2+|\p^\al\theta|^2$ for $0\le\la<1$ and $t>t_0$.
Consequently, integrating \eqref{3.16} over $[t_0,t]\times\R^d$ gives
\begin{align}
    &\quad (1+t)^{2\la}\|\p\p^\al\theta(t,\cdot)\|^2+\|\p^\al\theta(t,\cdot)\|^2
        +\int_{t_0}^t(1+s)^\la\|\p\p^\al\theta(s,\cdot)\|^2\,ds \no\\
    & \ls \mathcal{E}_4^2[\theta](t_0)+\int_{t_0}^t(1-\la)(1+s)^{\la-2}\|\theta(s,\cdot)\|^2\,ds
        +\int_{t_0}^t(1+s)^\la\|\p\p^{<\al}\theta(s,\cdot)\|^2\,ds \no\\
    & +\left|\int_{t_0}^t\int_{\R^d} \left(\p^\al Q_1(\theta,u)+\p^\al
        Q_2(\theta,u)\right) \left(2(1+s)^{2\la}\p_t\p^\al\theta+
        \mu(1+s)^\la\p^\al\theta\right) \,dxds\right|. \label{3.18}
\end{align}
Next we deal with the last term in the right hand side of \eqref{3.18}.
It follows from \eqref{3.1}-\eqref{3.2} and $|\al|\le3$ that
\begin{align}
    &\quad \left|\int_{t_0}^t\int_{\R^d} \p^\al Q_2(\theta,u) \left(2(1+s)^{2\la}
        \p_t\p^\al\theta+\mu(1+s)^\la\p^\al\theta\right) \,dxds\right| \no\\
    &\ls K_1\ve\int_{t_0}^t(1+s)^\la\|\p\p^{\le3}(\theta,u)(s,\cdot)\|^2\,ds \no\\
    &\ls K_1\ve\int_{t_0}^t(1+s)^\la\left(\|\p\p^{\le3}\theta(s,\cdot)\|^2
        +\|\p^{\le3} w(s,\cdot)\|^2\right)\,ds
        +K_1\ve\int_{t_0}^t (1+s)^{-\la}\|u(s,\cdot)\|^2\,ds. \label{3.19}
\end{align}
Now we turn our attention to the term $\p^\al Q_1(\theta,u)$. It is easy to get
\begin{align}
    \p^\al Q_1(\theta,u) \defeq -\p^\al\left(\frac\mu{(1+t)^\la}u\cdot\nabla\theta\right)+(\g-1)\theta\Delta\p^\al\theta \no\\
    -2u\cdot\nabla\p_t\p^\al\theta-\sum_{i,j=1}^d u_iu_j\p_{ij}^2\p^\al\theta+Q^\al_1(\theta,u). \label{3.20}
\end{align}
One easily checks that \eqref{3.19} still holds if $\p^\al Q_2(\theta,u)$ is replaced by $Q^\al_1(\theta,u)$.
In addition, for $\al=0$, we see that
\begin{align}
    &\quad \left|\int_{t_0}^t\int_{\R^d} \frac\mu{(1+s)^\la}u\cdot\nabla\theta
        \left(2(1+s)^{2\la}\p_t\theta+\mu(1+s)^\la\theta\right) \,dxds\right| \no\\
    &\ls \int_{t_0}^t \left(|\theta(s,\cdot)|_\infty \|u(s,\cdot)\|+(1+s)^\la
        |u(s,\cdot)|_\infty\|\p_t\theta(s,\cdot)\|\right) \|\nabla\theta(s,\cdot)\|\,ds \no\\
    &\ls K_1\ve\int_{t_0}^t \left( (1+s)^{-\la}\|u(s,\cdot)\|^2+(1+s)^\la
        \|\p\theta(s,\cdot)\|^2\right) \,ds, \label{3.21}
\end{align}
where we have used \eqref{3.1} again. If $\al>0$, by \eqref{3.2} we find that
\begin{align}
    &\quad \left|\int_{t_0}^t\int_{\R^d} \p^\al\left(\frac\mu{(1+s)^\la}u\cdot\nabla\theta\right)
        \left(2(1+s)^{2\la}\p_t\p^\al\theta+\mu(1+s)^\la\p^\al\theta\right) \,dxds\right| \no\\
    &\ls K_1\ve\int_{t_0}^t(1+s)^\la\left(\|\p\p^{\le3}\theta(s,\cdot)\|^2
        +\|\p^{\le3} w(s,\cdot)\|^2\right)\,ds+K_1\ve\int_{t_0}^t (1+s)^{-\la}\|u(s,\cdot)\|^2\,ds. \label{3.22}
\end{align}
On the other hand, direct computation derives the following identities
\begin{align*}
    (1+t)^{2\la}\theta\Delta\p^\al\theta\p_t\p^\al\theta=
    \opdiv\left[(1+t)^{2\la}\theta\nabla\p^\al\theta\p_t\p^\al\theta\right]
        -(1+t)^{2\la}\nabla\theta\cdot\nabla\p^\al\theta\p_t\p^\al\theta \\
    -\frac12\,\p_t\left[(1+t)^{2\la}\theta\,|\nabla\p^\al\theta|^2\right]
        +\la(1+t)^{2\la-1}\theta\,|\nabla\p^\al\theta|^2+\frac12(1+t)^{2\la}\p_t\theta\,|\nabla\p^\al\theta|^2
\end{align*}
and
\begin{equation*}
    (1+t)^\la\theta\Delta\p^\al\theta\p^\al\theta=\opdiv\left[(1+t)^\la
    \theta\nabla\p^\al\theta\p^\al\theta\right]-(1+t)^\la(\theta\,
    |\nabla\p^\al\theta|^2+\nabla\theta\cdot\nabla\p^\al\theta\p^\al\theta).
\end{equation*}
Integrating these two identities over $[t_0,t]\times\R^d$ yields
\begin{align}
    &\quad \left|\int_{t_0}^t\int_{\R^d} \theta\Delta\p^\al\theta \left(2
        (1+s)^{2\la}\p_t\p^\al\theta+\mu(1+s)^\la\p^\al\theta\right)\,dxds\right| \no\\
    & \ls K_1\ve\left(\mathcal{E}_4^2[\theta](t)+\mathcal{E}_4^2[\theta](t_0)\right)
        +K_1\ve\int_{t_0}^t(1+s)^\la\|\p\p^{\le3}\theta(s,\cdot)\|^2\,ds, \label{3.23}
\end{align}
where we have used \eqref{3.1} and $\la\le1$.

Analogously for the remaining items $u\cdot\nabla\p_t\p^\al\theta$
and $\ds\sum_{i,j=1}^du_iu_j\p_{ij}^2\p^\al\theta$ in $\p^\al Q_1(\theta,u)$ (see \eqref{3.20}),
direct computations show
\begin{align*}
    2(1+t)^{2\la}u\cdot\nabla\p_t\p^\al\theta\p_t\p^\al\theta=
    & \opdiv\left[(1+t)^{2\la}u\,|\p_t\p^\al\theta|^2\right]
    -(1+t)^{2\la}\opdiv u\,|\p_t\p^\al\theta|^2, \\
    (1+t)^\la u\cdot\nabla\p_t\p^\al\theta\p^\al\theta=
    & \opdiv\left[(1+t)^\la u\,\p_t\p^\al\theta\p^\al\theta\right]
    -(1+t)^\la\p_t\p^\al\theta(u\cdot\nabla\p^\al\theta+\opdiv u\,
        \p^\al\theta)
\end{align*}
and
\begin{align*}
    2(1+t)^{2\la}u_iu_j\p_{ij}^2\p^\al\theta\p_t\p^\al\theta=
    &\p_i\left[(1+t)^{2\la}u_iu_j\p_j\p^\al\theta\p_t\p^\al\theta\right]
        +\p_j\left[(1+t)^{2\la}u_iu_j\p_i\p^\al\theta\p_t\p^\al\theta\right] \\
    & -(1+t)^{2\la}\p_t\p^\al\theta\Big[\p_i(u_iu_j)\p_j\p^\al\theta+
        \p_j(u_iu_j)\p_i\p^\al\theta\Big] \\
    & -\p_t\left[(1+t)^{2\la}u_iu_j\p_i\p^\al\theta\p_j\p^\al\theta\right]
        +(1+t)^{2\la}\p_t(u_iu_j)\p_i\p^\al\theta\p_j\p^\al\theta \\
    & +2\la(1+t)^{2\la-1}u_iu_j\p_i\p^\al\theta\p_j\p^\al\theta, \\
        (1+t)^\la u_iu_j\p_{ij}^2\p^\al\theta\p^\al\theta=
    & \p_i\left[(1+t)^\la u_iu_j\p_j\p^\al\theta\p^\al\theta\right]
        -(1+t)^\la\p_j\p^\al\theta\p_i(u_iu_j\p^\al\theta).
\end{align*}
Then we have that
\begin{align}
    &\quad \left|\int_{t_0}^t\int_{\R^d} \left(2u\cdot\nabla\p_t\p^\al\theta
        +\sum_{i,j=1}^du_iu_j\p_{ij}^2\p^\al\theta\right) \left(2(1+s)^{2\la}\p_t\p^\al\theta
        +\mu(1+s)^\la\p^\al\theta\right) \,dxds\right| \no\\
    &\ls K_1\ve\left(\mathcal{E}_4^2[\theta](t)+\mathcal{E}_4^2[\theta](t_0)\right)
        +K_1\ve\int_{t_0}^t (1+s)^{-\la}\|u(s,\cdot)\|^2\,ds \no\\
    &\quad +K_1\ve\int_{t_0}^t(1+s)^\la\left(\|\p\p^{\le3}\theta(s,\cdot)\|^2
        +\|\p^{\le3} w(s,\cdot)\|^2\right)\,ds. \label{3.24}
\end{align}
Substituting \eqref{3.19}-\eqref{3.24} into \eqref{3.18} yields
\begin{align}
    &\quad (1+t)^{2\la}\|\p\p^\al\theta(t,\cdot)\|^2+\|\p^\al\theta(t,\cdot)\|^2
        +\int_{t_0}^t(1+s)^\la\|\p\p^\al\theta(s,\cdot)\|^2\,ds \no\\
    & \ls \mathcal{E}_4^2[\theta](t_0)+K_1\ve\,\mathcal{E}_4^2[\theta](t)
        +\int_{t_0}^t(1-\la)(1+s)^{\la-2}\|\theta(s,\cdot)\|^2\,ds+\int_{t_0}^t(1+s)^\la\|\p\p^{<\al}\theta(s,\cdot)\|^2\,ds \no\\
    &\quad +K_1\ve\int_{t_0}^t(1+s)^\la\left(\|\p\p^{\le3}\theta(s,\cdot)\|^2
        +\|\p^{\le3} w(s,\cdot)\|^2\right)\,ds +K_1\ve\int_{t_0}^t (1+s)^{-\la}\|u(s,\cdot)\|^2\,ds. \label{3.25}
\end{align}
Summing up \eqref{3.25} from $|\al|=0$ to $|\al|=3$ and applying Gronwall's inequality for $\la<1$ yield \eqref{3.15} provided that $\ve>0$ is small enough.
This completes the proof of Lemma~\ref{lem-thetaglobal}.
\end{proof}

\subsection{Proof of Theorem~\ref{thm1}.}
\begin{proof}[Proof  of Theorem~\ref{thm1}]
First, we assume that $\mathcal{E}_4[\theta,u](t) \le K_1\ve$ holds.
Multiplying \eqref{3.12} by $(1+t)^{2\la}$ yields
\begin{align}
    &\quad \frac{d}{dt}\|(1+t)^\la\p^{\le3}w(t,\cdot)\|^2+\left(\mu
        (1+t)^\la-2\la(1+t)^{2\la-1}\right)\|\p^{\le3}w(t,\cdot)\|^2 \no\\
    &\ls K_1\ve(1+t)^\la\|\p^{\le3}w(t,\cdot)\|\,\|\p\p^{\le3} u(t,\cdot)\|, \label{3.26}
\end{align}
where we have used \eqref{3.1}. In view of \eqref{3.17}, the second term on the first line in \eqref{3.26} is bounded below by $(1+t)^\la\|\p^{\le3}w(t,\cdot)\|^2$.
Integrating \eqref{3.8} and \eqref{3.26} over $[t_0,t]\times\R^d$ derives
\begin{equation}\label{3.27}
    \|(\theta,u)(t,\cdot)\|^2+\int_{t_0}^t (1+s)^{-\la}\|u(s,\cdot)\|^2\,ds
    \ls \|(\theta,u)(t_0,\cdot)\|^2+K_1\ve\int_{t_0}^t (1+s)^\la\|\nabla\theta(s,\cdot)\|^2\,ds
\end{equation}
and
\begin{align}
    &\quad \|(1+t)^\la\p^{\le3}w(t,\cdot)\|^2+\int_{t_0}^t (1+s)^\la\|\p^{\le3}w(s,\cdot)\|^2\,ds \no\\
    &\ls \|\p^{\le3}w(t_0,\cdot)\|^2+K_1\ve \int_{t_0}^t (1+s)^\la\|\p\p^{\le3} u(s,\cdot)\|^2\,ds. \label{3.28}
\end{align}
Collecting \eqref{3.15}, \eqref{3.27}-\eqref{3.28} with \eqref{3.2}, we
infer $\mathcal{E}_4[\theta,u](t) \le C_1\mathcal{E}_4[\theta,u](t_0)$.
It follows from the local existence of the hyperbolic system
that $\mathcal{E}_4[\theta,u](t_0) \le C_2\ve$.
Let $K_1=2C_1C_2$ and choose $\ve>0$ sufficiently small.
Then, we conclude that $\mathcal{E}_4[\theta,u](t) \le \frac12 K_1\ve$, which implies
that \eqref{euler-reform} admits a global solution $(\theta,u)$ for Case 1.
It follows from the definition of $\theta$ (i.e. \eqref{theta-def})
that there exists a global solution $(\rho,u)$ to \eqref{euler-eqn} for Case 1.

Next we show
\begin{equation}\label{vorticity-decay}
    \|\p^{\le3}w(t,\cdot)\| \ls \Xi(t)^{-\frac13} \,\ve,
\end{equation}
where $\Xi(t)=e^{\frac{\mu}{1-\la}[(1+t)^{1-\la}-1]}$ has been defined in \eqref{Xi-def}.

For this purpose, we assume that $\| \Xi(t)^\frac13 \p^{\le3}w(t,\cdot)\|\le K_2\ve$ holds for sufficiently large
constant $K_2>0$ and small $\ve>0$. This immediately implies
\begin{equation}\label{3.30}
\Xi(t)^\frac13 |\p^{\le1}w(t,\cdot)|_\infty \ls K_2\ve.
\end{equation}
Multiplying \eqref{3.12} by $ \Xi(t)^\frac23$ yields
\begin{align}
    &\frac{d}{dt} \|\Xi(t)^\frac13 \p^{\le3}w(t,\cdot)\|^2
        +\frac{\mu}{3(1+t)^\la} \,\|\Xi(t)^\frac13 \p^{\le3}w(t,\cdot)\|^2 \no\\
    &\ls \Xi(t)^\frac23 |\p^{\le1} w(t,\cdot)|_\infty\,
        \|\p^{\le3}w(t\cdot)\|\,\|\p\p^{\le3} u(t,\cdot)\|. \label{3.31}
\end{align}
Substituting \eqref{3.2} and \eqref{3.30} into \eqref{3.31} and applying Young's inequality, we then have
\begin{align}
    &\quad \frac{d}{dt} \|\Xi(t)^\frac13 \p^{\le3} w(t,\cdot)\|^2+
        \frac{\mu}{3(1+t)^\la}\|\Xi(t)^\frac13 \p^{\le3} w(t,\cdot)\|^2 \no\\
    &\ls \Xi(t)^\frac23 |\p^{\le1}w(t,\cdot)|_\infty \,\|\p^{\le3}w(t\cdot)\|
        \left((1+t)^{-\la}\|u(t,\cdot)\|+\|\p^{\le3}w(t\cdot)\|
        +\|\p\p^{\le3}\theta(t\cdot)\| \right) \no\\
    &\ls K_2\ve \left((1+t)^{-\la}\|u(t,\cdot)\|^2+\big((1+t)^{-\la}
        +\Xi(t)^{-\frac13}\big)\|\Xi(t)^\frac13 \p^{\le3} w(t,\cdot)\|^2
        +(1+t)^\la\|\p\p^{\le3}\theta(t,\cdot)\|^2 \right). \label{3.32}
\end{align}
Collecting \eqref{3.8}, \eqref{3.15}, \eqref{3.32} and applying the same argument as in the proof
for the global existence of $(\rho, u)$, we infer \eqref{vorticity-decay}.
This completes the proof of \eqref{vorticity-decay}. Thus the proof of Theorem~\ref{thm1} is completed.
\end{proof}

\begin{remark}\label{rmk3.2}
The proof of Theorem~\ref{thm1} can be applied to the case of $\la=1$, $\mu>2$.
In this case, Lemma~\ref{lem-velocity1}-\ref{lem-vorticity} still hold and the coefficient of
the first term in the second line of \eqref{3.16} is $(\mu-2)(1+t)$,
which plays the same role as \eqref{3.17}. Instead of the identity below \eqref{3.17}, we have
\begin{align*}
    (1+t)^2|\p\p^\al\theta|^2+\mu(1+t)\p^\al\theta
        \p_t\p^\al\theta+\frac{\mu(\mu-1)}{2}|\p^\al\theta|^2 \\
    =(1+t)^2\left(\frac{\mu-2}{3\mu-2}|\p_t\p^\al\theta|^2+|\nabla\p^\al\theta|^2\right)
        +\frac{\mu(\mu-2)}{8}|\p^\al\theta|^2 \\
    +\left((1+t)\sqrt\frac{2\mu}{3\mu-2}\p_t\p^\al\theta+\sqrt\frac{\mu(3\mu-2)}{8}\p^\al\theta\right)^2,
\end{align*}
which is equivalent to $(1+t)|\p\p^\al\theta|^2+|\p^\al\theta|^2$ for $\mu>2$.
However, we cannot obtain the exponential decay of the vorticty $w$.
\end{remark}


\section{Proof of Theorem 1.2.}\label{section4}


Theorem~\ref{thm2} in three space dimensions has been proved in \cite{HWY15}.
In this section, we fix $d=2$ and assume that
\begin{equation}\label{4.1}
    E_5[\theta,u](t) \le K_3\ve
\end{equation}
holds. By the finite propagation speed property of hyperbolic systems,
one easily knows that $(\theta, u)$ and their derivatives are supported in $\{(t,x)\colon |x|\le t+M\}$,
which implies that for $\al>0$,
\begin{equation}\label{4.2}
    |Z^\al(\theta,u)(t,x)| \ls (1+t) |\p Z^{<\al}(\theta,u)(t,x)|.
\end{equation}
On the other hand, collecting \eqref{Klainerman-ineq}-\eqref{weight-pointwise} with assumption \eqref{4.1} derives
the following pointwise estimate
\begin{equation}\label{4.3}
    |\sigma_-^{-\frac12}(t,\cdot)Z^{\le2}(\theta,u)(t,\cdot)|_\infty+
    |\sigma_-^\frac12(t,\cdot)\p Z^{\le2}(\theta,u)(t,\cdot)|_\infty
    \ls \frac{K_3\ve}{1+t}.
\end{equation}
To prove Theorem~\ref{thm2} for $d=2$, we shall focus on the following parts.

\subsection{Estimates of velocity $u$ and its derivatives.}

The following lemma is an application of Lemma~\ref{lem-divcurl} and \eqref{4.2}-\eqref{4.3}.
\begin{lemma}\label{lem-Zvelocity}
Under assumption \eqref{4.3}, for all $t\ge0$, it holds that
\begin{equation}\label{4.4}
    E_5[u](t)\ls (1+t)^{-\frac12}\|u(t,\cdot)\|+(1+t)^\frac12\|\p Z^{\le4}\theta(t,\cdot)\|.
\end{equation}
\end{lemma}
\begin{proof}
In view of $\opcurl u_0\equiv0$ and \eqref{2dvorticity-eqn}, it is easy to know that $\opcurl u(t,x)\equiv0$
always holds for any $t\ge0$ as long as the smooth solution $(\theta, u)$ of \eqref{euler-reform} exists.
Then, it follows from $\opcurl u\equiv0$ that
\begin{align*}
    \opcurl Z^\al u= Z^\al\opcurl u+\sum_{\beta<\al}C_{\al,\beta}\p Z^\beta u= \sum_{\beta<\al}C_{\al,\beta}\p Z^\beta u,
\end{align*}
which can be abbreviated as
\begin{equation}\label{4.5}
    \opcurl Z^\al u=\p Z^{<\al}u.
\end{equation}
Analogously, we have
\begin{equation}\label{4.6}
    \opdiv Z^\al u=Z^\al\opdiv u+\p Z^{<\al}u.
\end{equation}
Taking $U=Z^\al u$ with $|\al|\le4$ in \eqref{divcurl-estimate} and applying \eqref{4.5}-\eqref{4.6} yield
\begin{align}
    \|\nabla Z^\al u(t,\cdot)\|
    &\ls \|Z^\al\opdiv u(t,\cdot)\|+\|\p Z^{<\al}u(t,\cdot)\| \no\\
    &\ls \|\p Z^{\le\al}\theta(t,\cdot)\|+K_3\ve(1+t)^{-1}\sum_{0<\beta\le\al}
        \|Z^\beta(\theta,u)(t,\cdot)\|+\|\p Z^{<\al}u(t,\cdot)\| \no\\
    &\ls \|\p Z^{\le\al}\theta(t,\cdot)\|+\|\p Z^{<\al}u(t,\cdot)\|, \label{4.7}
\end{align}
where we have used the first equation in \eqref{euler-reform} and \eqref{4.2}-\eqref{4.3}.
On the other hand, one easily gets
\begin{align}
    \|\p_t Z^\al u(t,\cdot)\|
    &\ls \|Z^\al\p_t u(t,\cdot)\|+\|\p Z^{<\al}u(t,\cdot)\| \no\\
    &\ls \|\p Z^{\le\al}\theta(t,\cdot)\|+K_3\ve\|\p Z^{\le\al}u(t,\cdot)\|
        +(1+t)^{-1}\|u(t,\cdot)\|+\|\p Z^{<\al}u(t,\cdot)\|. \label{4.8}
\end{align}
Summing up \eqref{4.7}-\eqref{4.8} from $|\al|=0$ to $|\al|=4$, then \eqref{4.4} is
obtained by the smallness of $\ve$.
This completes the proof of Lemma~\ref{lem-Zvelocity}.
\end{proof}

\begin{lemma}\label{lem-velocity3}
Let $\mu>0$. Under assumption \eqref{4.3}, for all $t\ge0$, it holds that
\begin{equation}\label{4.9}
    \frac{d}{dt}\left[(1+t)^{-1}\|(\theta,u)(t,\cdot)\|^2\right]
        +\frac1{2(1+t)^2}\|(\theta,u)(t,\cdot)\|^2 \le 0.
\end{equation}
\end{lemma}
\begin{proof}
Multiplying the second equation in \eqref{euler-reform} by $(1+t)^{-1}u$ derives
\begin{equation}\label{4.10}
  \frac12\,\p_t\left[(1+t)^{-1}|u|^2\right]+\frac{\mu+\frac12}{(1+t)^2}|u|^2
    +(1+t)^{-1}u\cdot\nabla\theta=-\frac12\,(1+t)^{-1}u\cdot\nabla|u|^2.
\end{equation}
From the first equation in \eqref{euler-reform}, we see that
\begin{align}
    (1+t)^{-1}u\cdot\nabla\theta
    &= \opdiv\left[(1+t)^{-1}(\theta u+(\g-1)\theta^2 u)\right]+
        \frac12\,\p_t\left[(1+t)^{-1}|\theta|^2\right] \no\\
    &\quad +\frac1{2(1+t)^2}|\theta|^2+(3-2\g)(1+t)^{-1}\theta\,
        u\cdot\nabla\theta, \label{4.11}
\end{align}
which is similar to the expression in \eqref{3.10}.
Substituting \eqref{4.11} into \eqref{4.10} and integrating it over $\R^2$ yield
\begin{align}
    &\quad \frac{d}{dt}\left[(1+t)^{-1}\|(\theta,u)(t,\cdot)\|^2\right]+
        \frac{1}{2(1+t)^2}\|(\theta,u)(t,\cdot)\|^2 \no\\
    &\ls (1+t)^{-1}|\nabla\theta(t,\cdot)|_\infty\|u(t,\cdot)\|\,\|\theta(t,\cdot)\|
        +(1+t)^{-1}|\nabla u(t,\cdot)|_\infty\|u(t,\cdot)\|^2. \label{4.12}
\end{align}
Substituting \eqref{4.3} into \eqref{4.12}, then \eqref{4.9} can be obtained from the smallness of $\ve$.
This completes the proof of Lemma~\ref{lem-velocity3}.
\end{proof}

\subsection{Estimates of $\theta$ and its derivatives.}

The following lemma shows the estimates of $\theta$.
\begin{lemma}\label{lem-Ztheta}
Let $\mu>1$. Under assumption \eqref{4.3}, for all $t\ge0$, it holds that
\begin{align}
    &\quad E_5^2[\theta](t)+\int_0^t \Big(\|\p Z^{\le4}\theta(s,\cdot)\|^2
        +(1+s)^{-2}\|Z^{\le4}\theta(s,\cdot)\|\Big)\,ds \no\\
    & \ls E_5^2[\theta](0)+K_3\ve\int_0^t (1+s)^{-2}\|u(s,\cdot)\|^2 \,ds. \label{4.13}
\end{align}
\end{lemma}
\begin{proof}
Acting $Z^\al$ with $|\al|\le4$ on both sides of equation \eqref{theta-eqn} implies
\begin{equation}\label{4.14}
    \p_t^2Z^\al\theta+\frac\mu{1+t}\p_tZ^\al\theta-\Delta Z^\al\theta=Q^\al_2(\theta,u),
\end{equation}
where
\begin{align}
    Q^\al_2(\theta,u) &\defeq Z^\al Q(\theta,u)+Q^\al_{21}+Q^\al_{22}+
        Q^\al_{23} \no\\
    &\defeq Z^\al Q(\theta,u)+[\p_t^2-\Delta,Z^\al]\,\theta+
        \frac\mu{1+t}[\p_t,Z^\al]\,\theta
        +\sum_{0<\beta\le\al}Z^\beta\left(\frac\mu{1+t}\right)
        Z^{\al-\beta}\p_t\theta. \label{4.15}
\end{align}
Multiplying \eqref{4.14} by $2\mu(1+t)\p_t Z^\al\theta+(2\mu-1)Z^\al\theta$ yields
\begin{align}
    &\p_t\left[\mu(1+t)|\p Z^\al\theta|^2+(2\mu-1)Z^\al\theta\p_tZ^\al\theta+
        \frac{\mu(2\mu-1)}{2(1+t)}|Z^\al\theta|^2\right] \no\\
    &\quad +(\mu-1)(2\mu-1)|\p_tZ^\al\theta|^2+(\mu-1)|\nabla Z^\al\theta|^2+
        \frac{\mu(2\mu-1)}{2(1+t)^2}|Z^\al\theta|^2 \no\\
    &=\opdiv\left[\nabla Z^\al\theta \left(2\mu(1+t)\p_t Z^\al\theta+(2\mu-1)
        Z^\al\theta\right) \right]+Q^\al_2(\theta,u)\left(2\mu(1+t)\p_t Z^\al\theta+(2\mu-1)
        Z^\al\theta\right). \label{4.16}
\end{align}
Thanks to $\mu>1$, we see that for the term in the square bracket of the first line in \eqref{4.16}
\begin{align*}
    &\quad \mu(1+t)|\p Z^\al\theta|^2+(2\mu-1)Z^\al\theta\p_tZ^\al\theta+
        \frac{\mu(2\mu-1)}{2(1+t)}|Z^\al\theta|^2 \\
    &=(1+t)\Big(\frac12|\p_tZ^\al\theta|^2+\mu|\nabla Z^\al\theta|^2\Big)+
        \frac{(\mu-1)(2\mu-1)}{2(1+t)}|Z^\al\theta|^2+\frac{2\mu-1}{2(1+t)}
        \Big( (1+t)\p_tZ^\al\theta+Z^\al\theta\Big)^2,
\end{align*}
which is equivalent to $(1+t)|\p Z^\al\theta|^2+\frac1{1+t}|Z^\al\theta|^2$.
Integrating \eqref{4.16} over $[0,t]\times\R^2$ yields
\begin{align}
    &\quad (1+t)\|\p Z^\al\theta(t,\cdot)\|^2+(1+t)^{-1}\|Z^\al\theta(t,\cdot)\|^2
        +\int_0^t \Big(\|\p Z^\al\theta(s,\cdot)\|^2+(1+s)^{-2}
        \|Z^\al\theta(s,\cdot)\|\Big)\,ds \no\\
    &\ls E_5^2[\theta](0)+\left|\int_0^t\int_{\R^2} Q^\al_2(\theta,u)\Big(2\mu
        (1+s)\p_t Z^\al\theta+(2\mu-1)Z^\al\theta\right)\,dxds\Big|. \label{4.17}
\end{align}
It follows from a direct computation that
\begin{equation}\label{4.18}
    Q^\al_{21}=Z^{<\al}(\p_t^2-\Delta)\theta=Z^{<\al}Q(\theta,u)-Z^{<\al}\left(
        \frac{\mu}{1+t}\p_t\theta\right) \defeq Z^{<\al}Q(\theta,u)+Q^\al_{24}.
\end{equation}
For $\al=0$, we find that $Q^\al_{22}=Q^\al_{23}=Q^\al_{24}=0$.
For $\al>0$ and by \eqref{4.2}, we see that
\begin{align}
    &\quad \left|\int_0^t\int_{\R^2}\left(Q^\al_{22}+Q^\al_{23}+Q^\al_{24}\right)
        \Big(2\mu(1+s)\p_t Z^\al\theta+(2\mu-1)Z^\al\theta\Big)\,dxds\right| \no\\
    &\ls \int_0^t\|\p Z^{\le\al}\theta(s,\cdot)\|\,\|\p Z^{<\al}\theta(s,\cdot)\|\,ds. \label{4.19}
\end{align}
Recall the definition of $Q(\theta,u)$ in \eqref{Q-def}-\eqref{Q2-def} as follows
\begin{align*}
    Q(\theta,u) &=Q_1(\theta,u)+Q_2(\theta,u), \\
    Q_1(\theta,u) &=-\frac\mu{1+t}u\cdot\nabla\theta+(\g-1)\theta\Delta\theta
        -2u\cdot\nabla\p_t\theta-\sum_{i,j=1,2} u_iu_j\p_{ij}^2\theta, \\
    Q_2(\theta,u) &= -\sum_{i,j=1,2}u_i\p_iu_j\p_j\theta-\p_tu\cdot\nabla\theta+(1+(\g-1)\theta)
        (\sum_{i,j=1,2}\p_iu_j\p_ju_i+(\g-1)|\opdiv u|^2).
\end{align*}
Then it follows from \eqref{4.2}-\eqref{4.4} and $|\al|\le4$ that
\begin{align}
    &\quad \left|\int_0^t\int_{\R^2} Z^{\le\al}Q_2(\theta,u)\Big(2\mu(1+s)
        \p_t Z^\al\theta+(2\mu-1)Z^\al\theta\Big) \,dxds\right| \no\\
    &\ls \int_0^t \Big(|\theta(s,\cdot)|_\infty+(1+s)|\p Z^{\le2}(\theta,u)
        (s,\cdot)|_\infty\Big)\,\|\p Z^{\le4}(\theta,u)(s,\cdot)\|^2 \,ds \no\\
    &\ls K_3\ve\int_0^t \Big(\|\p Z^{\le4}\theta(s,\cdot)\|^2+(1+s)^{-2}
        \|u(s,\cdot)\|^2\Big)\,ds. \label{4.20}
\end{align}
Substituting \eqref{4.19}-\eqref{4.20} into \eqref{4.17} derives
\begin{align}
    &\quad (1+t)\|\p Z^\al\theta(t,\cdot)\|^2+(1+t)^{-1}\|Z^\al\theta(t,\cdot)\|^2
        +\int_0^t \Big(\|\p Z^\al\theta(s,\cdot)\|^2+(1+s)^{-2}
        \|Z^\al\theta(s,\cdot)\|\Big)\,ds \no\\
    &\ls E_5^2[\theta](0)+\int_0^t\|\p Z^{<\al}\theta(s,\cdot)\|^2\,ds
        +K_3\ve\int_0^t \Big(\|\p Z^{\le4}\theta(s,\cdot)\|^2+(1+s)^{-2}
        \|u(s,\cdot)\|^2\Big)\,ds \no\\
    &\quad +\left|\int_0^t \int_{\R^2} Z^{\le\al}Q_1(\theta,u)\Big(2\mu(1+s)
        \p_t Z^\al\theta+(2\mu-1)Z^\al\theta\Big)\,dxds\right|. \label{4.21}
\end{align}
Next we focus on the treatment of $Z^{\le\al}Q_1(\theta,u)$.
In view of $|\al|\le4$, we find that
\begin{equation*}
    Z^{\le\al}(\theta\Delta\theta)=\theta\Delta Z^\al\theta+\sum_{1\le|\beta|\le2}
        Z^\beta\theta\,\p^2Z^{\al-\beta}\theta+\sum_{|\beta|\ge3}
        Z^\beta\theta\,\p^2Z^{\al-\beta}\theta,
\end{equation*}
which can be abbreviated as
\begin{equation*}
    Z^{\le\al}(\theta\Delta\theta)=\theta\Delta Z^\al\theta+Z^{\le2}\theta
        \,\p^2Z^{\le3}\theta+Z^{\le4}\theta\,\p^2Z^{\le1}\theta.
\end{equation*}
From this and the definition of $Q_1(\theta,u)$, we see that
\begin{align*}
    &\quad Z^{\le\al}Q_1(\theta,u) \no\\
    &=-Z^{\le\al}\left(\frac\mu{1+t}u\cdot\nabla\theta\right)+(\g-1)
        \theta\Delta Z^\al\theta-2u\cdot\nabla\p_tZ^\al\theta
        -\sum_{i,j=1,2}u_iu_j\p_{ij}^2Z^\al\theta+Q_3(\theta,u),
\end{align*}
where
\begin{align*}
    Q_3(\theta,u) \defeq (Z^{\le2}\theta+Z^{\le2}u)\p^2Z^{\le3}\theta+
        (Z^{\le4}\theta+Z^{\le4}u)\p^2Z^{\le1}\theta \no\\
    +Z^{\le2}u\,Z^{\le2}u\,\p^2Z^{\le3}\theta+Z^{\le2}u\,Z^{\le4}u\,
        \p^2Z^{\le1}\theta.
\end{align*}
If $\al=0$, applying \eqref{4.3} yields
\begin{align}
    &\quad \left|\int_0^t\int_{\R^2} \frac\mu{1+s}u\cdot\nabla\theta
        \Big(2\mu(1+s)\p_t\theta+(2\mu-1)\theta\Big) \,dxds\right| \no\\
    &\ls \int_0^t \Big(|u(s,\cdot)_\infty\|\p\theta(s,\cdot)\|^2+(1+s)^{-1}
        |\theta(s,\cdot)|_\infty\|u(s,\cdot)\|\,\|\nabla\theta(s,\cdot)\|\Big) \,ds \no\\
    &\ls K_3\ve\int_0^t \Big(\|\p Z^{\le4}\theta(s,\cdot)\|^2+(1+s)^{-2}
        \|u(s,\cdot)\|^2\Big) \,ds. \label{4.22}
\end{align}
If $\al>0$, from \eqref{4.2}-\eqref{4.4} we see that
\begin{align*}
    &\quad \left\|Z^{\le\al}\left(\frac\mu{1+t}u\cdot\nabla\theta\right) \right\| \\
    &\ls (1+t)^{-1} |Z^{\le2}u(t,\cdot)|_\infty\|\p Z^{\le4}\theta(t,\cdot)
        \|+(1+t)^{-1} |\p Z^{\le2}\theta(t,\cdot)|_\infty\|Z^{\le4}u(t,\cdot)\| \\
    &\ls K_3\ve\,(1+t)^{-1}\|\p Z^{\le4}\theta(t,\cdot)\|+K_3\ve(1+t)^{-2}\|u(t,\cdot)\|,
\end{align*}
which derives
\begin{align}
    &\quad \left|\int_0^t\int_{\R^2} Z^{\le\al}\left(\frac\mu{1+s}u\cdot\nabla\theta\right)
        \Big(2\mu(1+s)\p_t Z^\al\theta+(2\mu-1)Z^\al\theta\Big)\,dxds\right| \no\\
    &\ls K_3\ve\int_0^t \Big(\|\p Z^{\le4}\theta(s,\cdot)\|^2+(1+s)^{-2}\|u(s,\cdot)\|^2\Big) \,ds. \label{4.23}
\end{align}
As in Lemma~\ref{lem-thetaglobal}, direct computation derives the following identities
\begin{align*}
    (1+t)\theta\Delta Z^\al\theta\p_tZ^\al\theta=
    & \opdiv\left[(1+t)\theta\nabla Z^\al\theta\p_tZ^\al\theta\right]-(1+t)
        \nabla\theta\cdot\nabla Z^\al\theta\p_tZ^\al\theta \\
    & -\frac12\,\p_t\left[(1+t)\theta\,|\nabla Z^\al\theta|^2\right]+\frac12\,
        \theta\,|\nabla Z^\al\theta|^2+\frac12(1+t)\p_t\theta\,|\nabla Z^\al\theta|^2, \\
    \theta\Delta Z^\al\theta Z^\al\theta=
    &\opdiv\left[\theta\nabla Z^\al\theta Z^\al\theta\right]-\theta\,
        |\nabla Z^\al\theta|^2-\nabla\theta\cdot\nabla Z^\al\theta Z^\al\theta,
\end{align*}
together with \eqref{4.2}-\eqref{4.3}, this yields
\begin{align}
    &\quad \left|\int_0^t\int_{\R^2} \theta\Delta Z^\al\theta \Big(2\mu(1+s)
        \p_t Z^\al\theta+(2\mu-1)Z^\al\theta\Big)\,dxds\right| \no\\
    &\ls K_3\ve \Big(E_5^2[\theta](0)+E_5^2[\theta](t)\Big)+K_3\ve\int_0^t
        \|\p Z^{\le4}\theta(s,\cdot)\|^2 \,ds. \label{4.24}
\end{align}
Analogously, we have
\begin{align*}
    2(1+t)u\cdot\nabla\p_tZ^\al\theta\p_tZ^\al\theta=
    & \opdiv\left[(1+t)u\,|\p_tZ^\al\theta|^2\right]-(1+t)\opdiv u\,|\p_tZ^\al\theta|^2, \\
    u\cdot\nabla\p_tZ^\al\theta Z^\al\theta=
    & \opdiv\left[u\,\p_tZ^\al\theta Z^\al\theta\right]
        -\p_tZ^\al\theta(u\cdot\nabla Z^\al\theta+\opdiv u\,Z^\al\theta)
\end{align*}
and
\begin{align*}
    2(1+t)u_iu_j\p_{ij}^2Z^\al\theta\p_tZ^\al\theta=
    &\p_i\Big[(1+t)u_iu_j\p_jZ^\al\theta\p_tZ^\al\theta\Big]+
        \p_j\Big[(1+t)u_iu_j\p_iZ^\al\theta\p_tZ^\al\theta\Big] \\
    & -(1+t)\p_tZ^\al\theta\Big[\p_i(u_iu_j)\p_jZ^\al\theta+
        \p_j(u_iu_j)\p_iZ^\al\theta\Big] \\
    & -\p_t\Big[(1+t)u_iu_j\p_iZ^\al\theta\p_jZ^\al\theta\Big]+
        u_iu_j\p_iZ^\al\theta\p_jZ^\al\theta \\
    & +(1+t)\p_t(u_iu_j)\p_iZ^\al\theta\p_jZ^\al\theta, \\
    u_iu_j\p_{ij}^2Z^\al\theta Z^\al\theta=
    & \p_i\left[u_iu_j\p_jZ^\al\theta Z^\al\theta\right]-
        \p_jZ^\al\theta\p_i(u_iu_jZ^\al\theta),
\end{align*}
which derives
\begin{align}
    &\quad \left|\int_0^t\int_{\R^2} \left(2u\cdot\nabla\p_tZ^\al\theta+
        \sum_{i,j=1,2} u_iu_j\p_{ij}^2Z^\al\theta\right)
        \Big(2\mu(1+s)\p_t Z^\al\theta+(2\mu-1)Z^\al\theta\Big) \,dxds\right| \no\\
    &\ls K_3\ve \Big(E_5^2[\theta](0)+E_5^2[\theta](t)\Big)+K_3\ve
        \int_0^t \Big(\|\p Z^{\le4}\theta(s,\cdot)\|^2+(1+s)^{-2}
        \|u(s,\cdot)\|^2\Big) \,ds. \label{4.25}
\end{align}
Next we turn our attention to $Q_3(\theta,u)$.
Note that $Q_3(\theta,u)=0$ when $\al=0$.
It follows from direct calculation that for any function $\Phi(t,x)$
\begin{equation*}
    |\sigma_-(t,x)\p\Phi(t,x)| \ls |Z\Phi(t,x)|.
\end{equation*}
From this, \eqref{weight-L2} and \eqref{4.2}-\eqref{4.4}, we see that
\begin{align*}
    &\quad \|Q_3(\theta,u)\| \no\\
    &\ls \|Z^{\le2}(\theta,u)\p^2Z^{\le3}\theta\|+
        \|Z^{\le4}(\theta,u)\p^2Z^{\le1}\theta\| \\
    &\ls |\sigma_-^{-1}(t,\cdot)Z^{\le2}(\theta,u)(t,\cdot)|_\infty
        \|\sigma_-(t,\cdot)\p^2Z^{\le3}\theta(t,\cdot)\| \\
    &\quad +|\sigma_-^{\frac32}(t,\cdot)\p^2Z^{\le1}\theta(t,\cdot)|_\infty
        \|\sigma_-^{-\frac32}(t,\cdot)Z^{\le4}(\theta,u)(t,\cdot)\| \\
    &\ls K_3\ve\,(1+t)^{-1}\|\p Z^{\le4}\theta(t,\cdot)\|+|\sigma_-^{\frac12}(t,\cdot)
        \p Z^{\le2}\theta(t,\cdot)|_\infty \|\nabla Z^{\le4}(\theta,u)(t,\cdot)\| \\
    &\ls K_3\ve\,(1+t)^{-1}\|\p Z^{\le4}\theta(t,\cdot)\|+K_3\ve\,
        (1+t)^{-2}\|u(t,\cdot)\|,
\end{align*}
which implies
\begin{align}\label{}
    &\quad \left|\int_0^t\int_{\R^2} Q_3(\theta,u) \Big(2\mu(1+s)\p_t Z^\al\theta
        +(2\mu-1)Z^\al\theta\Big) \,dxds\right| \no\\
    &\ls K_3\ve\int_0^t \Big(\|\p Z^{\le4}\theta(s,\cdot)\|^2+(1+s)^{-2}\|u(s,\cdot)\|^2\Big) \,ds. \label{4.26}
\end{align}
Substituting \eqref{4.22}-\eqref{4.26} into \eqref{4.21} derives
\begin{align}
    &\quad (1+t)\|\p Z^\al\theta(t,\cdot)\|^2+(1+t)^{-1}\|Z^\al\theta(t,\cdot)\|^2
        +\int_0^t \Big(\|\p Z^\al\theta(s,\cdot)\|^2+(1+s)^{-2}
        \|Z^\al\theta(s,\cdot)\|\Big)\,ds \no\\
    &\ls E_5^2[\theta](0)+K_3\ve\,E_5^2[\theta](t)+\int_0^t\|\p Z^{<\al}\theta(s,\cdot)\|^2\,ds \no\\
    &\quad +K_3\ve\int_0^t \Big(\|\p Z^{\le4}\theta(s,\cdot)\|^2+(1+s)^{-2}\|u(s,\cdot)\|^2\Big) \,ds. \label{4.27}
\end{align}
Summing up \eqref{4.27} from $|\al|=0$ to $|\al|=4$ yields \eqref{4.13}.
This completes the proof of Lemma~\ref{lem-Ztheta}.
\end{proof}

\subsection{Proof of Theorem~\ref{thm2}.}
\begin{proof}[Proof of Theorem~\ref{thm2}]
Integrating \eqref{4.9} over $[0,t]$ yields that
\begin{equation*}
    (1+t)^{-1}\|(\theta,u)(t,\cdot)\|^2+\int_0^t (1+s)^{-2}\|(\theta,u)(s,\cdot)\|^2\,ds \ls \|(\theta,u)(0,\cdot)\|^2.
\end{equation*}
Collecting this with \eqref{4.4} and \eqref{4.13}, we conclude that $E_5[\theta,u](t) \le C_3\ve$.
Let $K_3=2C_3$, and choose $\ve>0$ sufficiently small.
Then, we infer $E_5[\theta,u](t) \le \frac12 K_3\ve$, which implies that \eqref{euler-reform}
admits a global solution for Case 2
with $\opcurl u_0(x)\equiv 0$. Thus we complete the proof of Theorem~\ref{thm2}.
\end{proof}


\section{Proof of Theorem~\ref{thm3}.}\label{section5}


In this section, we shall only prove Theorem~\ref{thm3} for $d=2$ since
the corresponding blowup result for Case 4 with $\opcurl u_0(x)\equiv 0$
in three space dimensions has been proved in \cite{HWY15}.
\begin{proof}[Proof of Theorem~\ref{thm3}]
We divide the proof into two parts.

\subsubsection*{Part~\RN{1}: \boldmath $\g=2$.}

Let $(\rho, u)$ be a $C^{\infty}-$smooth solution of \eqref{euler-eqn}.
For $l>0$, we define
\begin{equation}\label{5.1}
P(t,l)=\int_{x_1>l}\eta(x,l)\left(\rho(t,x)-\bar\rho\right)dx,
\end{equation}
where
\begin{equation*}
\eta(x,l)=(x_1-l)^2.
\end{equation*}
Employing the first equation in \eqref{euler-eqn} and an integration by parts, we see that
\begin{align}\label{5.2}
  \p_tP(t,l)&=\int_{x_1>l}\eta(x,l)\p_t\left(\rho(t,x)-\bar\rho\right)dx=
    -\,\int_{x_1>l}\eta(x,l)\opdiv(\rho u)(t,x)\,dx \no\\
  &=\int_{x_1>l}(\p_{x_1}\eta)(x,l)(\rho u_1)(t,x)\,dx,
\end{align}
where we have used the facts of $\eta(x,l)=0$ on $x_1=l$ and $u(t,x)=0$ for $|x|\ge t+M$.
By differentiating $\p_tP(t,l)$ in \eqref{5.2} again and using the equation of $u_1$ in
\eqref{euler-eqn}, we find that
\begin{multline*}
    \p_t^2P(t,l) =\int_{x_1>l}(\p_{x_1}\eta)(x,l)\p_t(\rho u_1)(t,x)\,dx
        =-\sum_{j=1,2}\int_{x_1>l}(\p_{x_1}\eta)\,\p_{x_j}(\rho u_1u_j)(t,x)\,dx \\
    -\int_{x_1>l}(\p_{x_1}\eta)(x,l)\p_{x_1}(p(t,x)-\bar p)\,dx-\frac\mu{(1+t)^\la}
        \int_{x_1>l}(\p_{x_1}\eta)(x,l)(\rho u_1)(t,x)\,dx,
\end{multline*}
where $\bar p=p(\bar\rho)$.
It follows from the integration by parts that
\begin{align}\label{5.3}
  \p_t^2P(t,l)+ \frac\mu{(1+t)^\la}\,\p_tP(t,l)&=\sum_{j=1,2}
  \int_{x_1>l}(\p_{x_1x_j}^2\eta)\rho u_1u_j\,dx+\int_{x_1>l}2(p-\bar p)\,dx \no\\
  &= \int_{x_1>l}2\rho u_1^2\,dx+\int_{x_1>l}2(p-\bar p)\,dx,
\end{align}
here we have used that $\p_{x_1}\eta(x,l)=0$ on $x_1=l$ and $p(t,x)-\bar p$ vanishes for $|x|\ge t+M$.
Note that
\begin{equation*}
\p_l^2\eta(x,l)=\Delta_x\eta(x,l)=2.
\end{equation*}
Then we have
\begin{equation}\label{5.4}
\int_{x_1>l}2(p-\bar p)\,dx=\int_{x_1>l}\p_l^2\eta(x,l)(p(t,x)-\bar p)\,dx=
\p_l^2\int_{x_1>l}\eta(x,l)(p(t,x)-\bar p)\,dx,
\end{equation}
where we have used the fact that $\eta$ and $\p_l\eta$ vanish on $x_1=l$.
Collecting \eqref{5.3}-\eqref{5.4}, we arrive at
\begin{equation}\label{5.5}
\p_t^2P(t,l)-\p_l^2P(t,l)+\frac\mu{(1+t)^\la}\,\p_tP(t,l)\defeq f(t,l)=\int_{x_1>l}2\rho u_1^2\,dx+G(t,l)\ge G(t,l),
\end{equation}
where
\begin{equation}\label{5.6}
    G(t,l)=\int_{x_1>l}2\left(p-\bar p-(\rho-\bar\rho)\right)dx
    =\p_l^2\int_{x_1>l}\eta(x,l)\left(p-\bar p-(\rho-\bar\rho)\right)dx
    \defeq \p_l^2\t G(t,l).
\end{equation}
Due to $\g=2$ and the sound speed $\bar c=\sqrt{2A\bar\rho}=1$, we have
\begin{equation}\label{5.7}
p-\bar p-(\rho-\bar\rho)=A\left(\rho^2-\bar\rho^2
-2\bar\rho\left(\rho-\bar\rho\right)\right)=A(\rho-\bar\rho)^2.
\end{equation}
Substituting \eqref{5.7} into \eqref{5.6} gives
\begin{equation*}
G(t,l),\,\t G(t,l) \ge 0.
\end{equation*}
For $M_0$ satisfying the condition \eqref{+condition}, let $\Sigma\defeq \{(t,l)\colon t\ge0, t+M_0\le l\le t+M\}$ be the strip domain.
By applying Riemann's representation (see \cite[\S5.5 of Chapter~5]{CH}) with the assumptions \eqref{q0-positive}-\eqref{+condition},
we have the following lower bound of the solution $P(t,l)$ to \eqref{5.5} for $(t,l)\in\Sigma$
\begin{equation}\label{5.8}
    P(t,l)\ge \frac14\Xi(t)^{-\frac12}q_0(l-t)+\frac14\int_0^t\int_{l-t+\tau}^{l+t-\tau} \left(\frac{\Xi(\tau)}{\Xi(t)}\right)^\frac12 f(\tau,y)\,dyd\tau.
\end{equation}
We put the proof of \eqref{5.8} in Appendix. Define the function
\begin{equation}\label{5.9}
    F(t)\defeq\int_0^t(t-\tau)\int_{\tau+M_0}^{\tau+M}P(\tau,l)\,\frac{dl}{\sqrt l} d\tau.
\end{equation}
From the definition of $\Xi(t)$, i.e., \eqref{Xi-def} for $\la=1$, $\mu\le1$ or $\la>1$,
we have $\Xi(t)^{-\frac12}\gt(t+M)^{-\frac12}$ and $\frac{\Xi(\tau)}{\Xi(t)}\gt\frac{\tau+M}{t+M}$.
Then, by \eqref{5.8}, we arrive at
\begin{align}
    &F''(t)=\int_{t+M_0}^{t+M}P(t,l)\,\frac{dl}{\sqrt l} \gt
        (t+M)^{-\frac12}\int_{t+M_0}^{t+M} q_0(l-t)\,\frac{dl}{\sqrt l} \no\\
    &\quad +\int_{t+M_0}^{t+M}\int_0^t\int_{l-t+\tau}^{l+t-\tau} \left(\frac{\tau+M}{t+M}\right)^\frac12 G(\tau,y)\,dyd\tau
            \frac{dl}{\sqrt l} \defeq J_1+J_2. \label{5.10}
\end{align}
From assumption \eqref{q0-positive}, we see that
\begin{equation}\label{5.11}
    J_1\gt \frac{1}{t+M}\int_{t+M_0}^{t+M}q_0(l-t)\,dl=\frac{1}{t+M}\int_{M_0}^M q_0(l)\,dl \gt \frac{\ve}{t+M}.
\end{equation}
To bound $J_2$ from below, we write
\begin{align}
  J_2&=\int_0^{t-M_1}\int_{\tau+M_0}^{\tau+M} \left(\frac{\tau+M}{t+M}\right)^\frac12
    G(\tau,y) \int_{t+M_0}^{y+t-\tau} \,\frac{dl}{\sqrt l}dyd\tau \no\\
  &\quad+\int_{t-M_1}^t\int_{\tau+M_0}^{2t-\tau+M_0} \left(\frac{\tau+M}{t+M}\right)^\frac12
    G(\tau,y) \int_{t+M_0}^{y+t-\tau} \,\frac{dl}{\sqrt l}dyd\tau \no\\
  &\quad+\int_{t-M_1}^t\int_{2t-\tau+M_0}^{\tau+M} \left(\frac{\tau+M}{t+M}\right)^\frac12
    G(\tau,y) \int_{y-t+\tau}^{y+t-\tau} \,\frac{dl}{\sqrt l}dyd\tau \no\\
  &\defeq J_{2,1}+J_{2,2}+J_{2,3}, \label{5.12}
\end{align}
where $M_1=\left(M-M_0\right)/2$. For $t<M_1$, $t-M_1$ in the limits of integration will be replaced by $0$.
For the integrand in $J_{2,1}$ we have that
\begin{equation}\label{5.13}
  \int_{t+M_0}^{y+t-\tau} \frac{dl}{\sqrt l}
  \gt \frac{y-\tau-M_0}{(t+M)^\frac12}
  \gt \frac{(t-\tau)(y-\tau-M_0)^2}{(t+M)^\frac32}.
\end{equation}
Analogously, for the integrands in $J_{2,2}$ and $J_{2,3}$ we have that
\begin{equation}\label{5.14}
  \int_{t+M_0}^{y+t-\tau} \frac{dl}{\sqrt l}
  \gt \frac{(t-\tau)(y-\tau-M_0)^2}{(t+M)^\frac32}
\end{equation}
and
\begin{equation}\label{5.15}
  \int_{y-t+\tau}^{y+t-\tau} \frac{dl}{\sqrt l}
  \gt \frac{t-\tau}{(t+M)^\frac12}
  \gt \frac{(t-\tau)(y-\tau-M_0)^2}{(t+M)^\frac32}.
\end{equation}
Substituting \eqref{5.13}-\eqref{5.15} into \eqref{5.12} yields
\begin{align*}
  J_2 \gt \frac{1}{(t+M)^2}\int_0^t (t-\tau)(\tau+M)^\frac12
  \int_{\tau+M_0}^{\tau+M} (y-\tau-M_0)^2\p_y^2\t G(\tau,y)\,dyd\tau,
\end{align*}
where $\t G(\tau,y)=\int_{x_1>y} (x_1-y)^2 \left(p(\tau,x)-\bar p-(\rho(\tau,x)-\bar\rho)\right)dx$.
Note that $\t G(\tau,y)=\p_y\t G(\tau,y)=0$ for $y=\tau+M$.
Then it follows from the integration by parts together with \eqref{5.6}-\eqref{5.7} that
\begin{align}
    J_2&\gt \frac{1}{(t+M)^2}\int_0^t (t-\tau)(\tau+M)^\frac12 \int_{\tau+M_0}^{\tau+M}\t G(\tau,y)\,dyd\tau \no \\
    &\gt \frac{1}{(t+M)^2}\int_0^t (t-\tau)(\tau+M)^\frac12 \int_{\tau+M_0}^{\tau+M} \int_{x_1>y} (x_1-y)^2
        \left(\rho(\tau,x)-\bar\rho\right)^2dxdyd\tau \no\\
    &\defeq \frac{c}{(t+M)^2}\,J_3. \label{5.16}
\end{align}
By applying the Cauchy-Schwartz inequality to $F(t)$ defined by \eqref{5.9}, we arrive at
\begin{equation}\label{5.17}
  F^2(t) \le J_3\int_0^t (t-\tau)(\tau+M)^{-\frac12}\int_{\tau+M_0}^{\tau+M}
  \int_{\t\Omega} (x_1-y)^2 \,dx\frac{dy}{y}d\tau\defeq J_3J_4,
\end{equation}
where $\t\Omega \defeq \{x\colon x_1>y,~|x|<\tau+M\}$.
Note that
\begin{align}
  J_4 &\ls \int_0^t (t-\tau)(\tau+M)^{-\frac12}\int_{\tau+M_0}^{\tau+M}\int_y^{\tau+M}
    (x_1-y)^2 [(\tau+M)^2-x_1^2]^\frac12 \,dx_1\frac{dy}{y}d\tau \no\\
  &\ls \int_0^t(t-\tau)\int_{\tau+M_0}^{\tau+M}(\tau+M-y)^\frac72\frac{dy}{y}d\tau \no \\
  &\ls \int_0^t(t-\tau)\int_{\tau+M_0}^{\tau+M}\frac{dy}{y}d\tau \no \\
  &\ls \int_0^t\frac{t-\tau}{\tau+M}\,d\tau \ls (t+M)\log(t/M+1). \label{5.18}
\end{align}
Combining \eqref{5.10}-\eqref{5.11} and \eqref{5.16}-\eqref{5.18} gives the following ordinary differential inequalities
\begin{align}
    F''(t) &\gt \frac{\ve}{t+M}, && t\ge0,  \label{5.19} \\
    F''(t) &\gt \left[(t+M)^3\log(t/M+1)\right]^{-1} \,F^2(t), && t\ge0.  \label{5.20}
\end{align}
Next, we apply \eqref{5.19}-\eqref{5.20} to prove that the lifespan $T_\ve$ of smooth solution $F(t)$
is finite for all $0<\ve\le\ve_0$.
The fact that $F(0)=F'(0)=0$, together with \eqref{5.19}, yields
\begin{align}
    F'(t) &\gt \ve\log(t/M+1), && t\ge0, \label{5.21} \\
    F(t)  &\gt \ve(t+M)\log(t/M+1), && t\ge t_1\defeq Me^2. \label{5.22}
\end{align}
Substituting \eqref{5.22} into \eqref{5.20} derives
\begin{equation*}
    F''(t) \gt \ve^2(t+M)^{-1}\log(t/M+1), \qquad t\ge t_1,
\end{equation*}
which leads to the improvement
\begin{equation}\label{5.23}
    F(t) \gt \ve^2(t+M)\log^2(t/M+1), \qquad t\ge t_2 \defeq Me^3>t_1.
\end{equation}
Substituting this into \eqref{5.20} yields
\begin{equation}\label{5.24}
    F''(t) \gt \ve^2(t+M)^{-2}\log(t/M+1)\,F(t), \qquad t\ge t_2.
\end{equation}
It follows from \eqref{5.21} that $F'(t)\ge0$ for $t\ge0$.
Then multiplying \eqref{5.24} by $F'(t)$ and integrating from $t_3$ (which will be chosen later) to $t$ derive
\begin{equation*}
    F'(t)^2 \ge F'(t_3)^2+C_4\ve^2\int_{t_3}^t (s+M)^{-2}\log(s/M+1)\,[F(s)^2]'ds.
\end{equation*}
It follows from the integration by parts that
\begin{multline}\label{5.25}
    F'(t)^2 \ge
    F'(t_3)^2+C_4\ve^2 \left((t+M)^{-2}\log(t/M+1)F(t)^2-(t_3+M)^{-2}\log(t_3/M+1)F(t_3)^2\right)\\
    -C_4\ve^2\int_{t_3}^t \left(\frac{\log(s/M+1)}{(s+M)^2}\right)' F(s)^2\,ds, \quad  t\ge t_3,
\end{multline}
where $\left(\frac{\log(s/M+1)}{(s+M)^2}\right)'\le0$ for $s\ge t_3\ge t_2$.
Since $F''(t)\ge 0$ and $F(0)=0$, the mean value theorem yields
\begin{equation}\label{5.26}
    F(t_3)=\int_0^{t_3}F'(s)ds \le t_3F'(t_3).
\end{equation}
Choose
\begin{equation}\label{5.27}
    t_3=Me^\frac1{2C_4\ve^2}-M,
\end{equation}
which satisfies $C_4\ve^2\log(t_3/M+1)=\frac12$.
Together with \eqref{5.25}-\eqref{5.26}, this yields
\begin{equation}\label{5.28}
    F'(t) \ge \sqrt{C_4}\ve(t+M)^{-1}\log^\frac12(t/M+1)\,F(t), \quad t\ge t_3.
\end{equation}
By integrating \eqref{5.28} from $t_3$ to $t$, we arrive at
\begin{equation*}
    \log\frac{F(t)}{F(t_3)} \ge \sqrt{C_4}\ve\log^\frac32\left(\frac{t+M}
        {t_3+M}\right), \quad t\ge t_3.
\end{equation*}
If $t\ge t_4\defeq Ct_3^2$, then we have
\begin{equation*}
    \log\frac{F(t)}{F(t_3)} \ge 8\log(t+M).
\end{equation*}
Together with \eqref{5.23} for $F(t_3)$, this yields
\begin{equation}\label{5.29}
    F(t) \gt \ve^2(t+M)^8, \quad t\ge t_4.
\end{equation}
Substituting this into \eqref{5.20} derives
\begin{equation*}
    F''(t) \gt \ve F(t)^\frac32, \quad t\ge t_4.
\end{equation*}
Multiplying this differential inequality by $F'(t)$ and integrating from $t_4$ to $t$ yield
\begin{equation*}
    F'(t)^2 \gt \ve\left(F(t)^\frac52-F(t_4)^\frac52\right).
\end{equation*}
On the other hand, $F(t)\ge 0$, $F''(t)\ge 0$, \eqref{5.26} and the mean value theorem imply that, for $t\ge t_4$,
\begin{equation*}
    F(t)=F'(\xi)(t-t_4)+F(t_4) \ge F'(t_4)(t-t_4) \ge F(t_4)\frac{t-t_4}{t_4},
\end{equation*}
where $t_4\le\xi\le t$. For $t\ge t_5\defeq Ct_4$, we have
\begin{equation*}
    F(t)^\frac52-F(t_4)^\frac52 \ge \frac{1}{2}F(t)^\frac52.
\end{equation*}
Thus
\begin{equation}\label{5.30}
    F'(t) \gt \sqrt\ve F(t)^\frac54, \quad t\ge t_5.
\end{equation}
If $T_\ve>2t_5$, then integrating \eqref{5.30} from $t_5$ to $T_\ve$ derives
\begin{equation*}
    F(t_5)^{-\frac14}-F(T_\ve)^{-\frac14} \gt \sqrt\ve T_\ve.
\end{equation*}
We see from \eqref{5.29} and $t_5=Ct_3^2$ that
\begin{equation*}
    F(t_5)\gt \ve^2e^\frac{C}{\ve^2},
\end{equation*}
which together with $F(T_\ve)>0$ is a contradiction. Thus, $T_\ve\le 2t_5=Ct_3^2$.
From the choice of $t_3$ in \eqref{5.27}, we see that $T_\ve\le e^{C/\ve^2}$.


\subsubsection*{Part~\RN{2}: \boldmath $\g>1$ and $\g\not=2$.}

In view of $\bar c=\sqrt{\g A\bar\rho^{\g-1}}=1$, instead of \eqref{5.7} we have
\begin{equation*}
    p-\bar p-(\rho-\bar\rho)=
    A\left(\rho^\g-\bar\rho^\g-\g\bar\rho^{\g-1}(\rho-\bar\rho)\right)
    \defeq A\psi(\rho,\bar\rho).
\end{equation*}
The convexity of $\rho^\g$ for $\g>1$ implies that $\psi(\rho,\bar\rho)$ is positive for $\rho\neq\bar\rho$.
Applying Taylor's theorem, we have
\begin{equation*}
\psi(\rho,\bar\rho) \ge C_{\g,\bar\rho} \,\Phi_\g(\rho,\bar\rho),
\end{equation*}
where $C_{\g,\bar\rho}$ is a positive constant and $\Phi_\g$ is given by
\begin{equation*}
\Phi_\g(\rho,\bar\rho)=
  \begin{cases}
    (\bar\rho-\rho)^\g, & \rho< \frac12\bar\rho,\\
    (\rho-\bar\rho)^2, & \frac12\bar\rho\le\rho\le2\bar\rho,\\
    (\rho-\bar\rho)^\g, & \rho>2\bar\rho.\\
  \end{cases}
\end{equation*}
For $\g>2$, we have that
$(\bar\rho-\rho)^\g=(\bar\rho-\rho)^2(\bar\rho-\rho)^{\g-2}\ge C_{\g,\bar\rho} (\rho-\bar\rho)^2$ for $2\rho<\bar\rho$ and
$(\rho-\bar\rho)^\g=(\rho-\bar\rho)^2(\rho-\bar\rho)^{\g-2}\ge C_{\g,\bar\rho} (\rho-\bar\rho)^2$ for $\rho> 2\bar\rho$.
Thus, $\Phi_\g(\rho,\bar\rho) \ge C_{\g,\bar\rho} (\rho-\bar\rho)^2$.
In this case, Theorem~\ref{thm3} can be shown completely analogously to Part~\RN{1}.

Next we treat the case $1<\g<2$. We define $F(t)$ as in \eqref{5.9}
\begin{equation*}
    F(t)\defeq \int_0^t(t-\tau)\int_{\tau+M_0}^{\tau+M} \int_{x_1>l}
    (x_1-l)^2\left(\rho(\tau,x)-\bar\rho\right)\,dx\frac{dl}{\sqrt l}d\tau.
\end{equation*}
Similarly to Part~\RN{1}, we have
\begin{equation}\label{5.31}
    F''(t)\ge J_1+J_2,
\end{equation}
where
\begin{align*}
    J_1 &\gt \frac{\ve}{t+M},\\
    J_2 &\gt (t+M)^{-2}\t J_3
\end{align*}
and
\begin{equation*}
    \t J_3 =\int_0^t(t-\tau)(\tau+M)^\frac12 \int_{\tau+M_0}^{\tau+M}
    \int_{x_1>y}(x_1-y)^2\,\Phi_\g(\rho(\tau,x), \bar\rho)\,dxdyd\tau.
\end{equation*}
Denote $\Omega_1=\{(\tau,x)\colon \bar\rho\le\rho(\tau,x)\le2\bar\rho\}$,
$\Omega_2=\{(\tau,x)\colon \rho(\tau,x)>2\bar\rho\}$, and
$\Omega_3=\{(\tau,x)\colon \rho(\tau,x)<\bar\rho\}$.
Divide $F(t)$ into the following three integrals over the domains $\Omega_i$ $(1\le i\le 3)$
\begin{equation*}
    F(t)=F_1(t)+F_2(t)+F_3(t)\defeq \int_{\Omega_1}\cdots+\int_{\Omega_2}\cdots+\int_{\Omega_3}\cdots.
\end{equation*}
Corresponding to the three parts of $F(t)$, we define $\t J_3\defeq\t J_{3,1}+\t J_{3,2}+\t J_{3,3}$.
In view of $F(t)\ge0$ and $F_3(t)\le0$, we have
\begin{equation*}
    F(t)\le F_1(t)+F_2(t).
\end{equation*}
Applying H\"{o}lder's inequality for the domains $\Omega_1$ and $\Omega_2$, we obtain that
\begin{align*}
    F(t)&\le \t J_{3,1}^\frac12\left(\int_0^t(t-\tau)(\tau+M)^{-\frac12}
        \int_{\tau+M_0}^{\tau+M}\frac1y\int_{\t\Omega}(x_1-y)^2\,dxdyd\tau\right)^\frac12 \\
    &\quad +\t J_{3,2}^\frac1\g\left(\int_0^t(t-\tau)(\tau+M)^{-\frac1{2(\g-1)}}
        \int_{\tau+M_0}^{\tau+M}\frac{1}{y^\frac{\g}{2(\g-1)}}\int_{\t\Omega}
        (x_1-y)^2\,dxdyd\tau\right)^\frac{\g-1}\g \\
    &\ls \t J_3^\frac12(t+M)^\frac12\log^\frac12 (t/M+1)
        +\t J_3^\frac1\g(t+M)^\frac{\g-1}\g \\
    &=\left(\t J_3(t+M)^{-1}\right)^\frac12(t+M)\log^\frac12 (t/M+1)
        +\left(\t J_3(t+M)^{-1}\right)^\frac1\g(t+M).
\end{align*}
In view of $1<\g<2$, we have $\ds\frac1{2\g}<\frac12<\frac1\g$.
Applying Young's inequality yields
\begin{equation*}
    F(t) \ls \Big(\big(\t J_3(t+M)^{-1}\big)^\frac1{2\g}+
    \big(\t J_3(t+M)^{-1}\big)^\frac{1}{\g}\Big)(t+M)\log^\frac12 (t/M+1),
    \quad t\ge \t t_1\defeq Me.
\end{equation*}
Together with the fact that $F(t)\gt \ve(t+M)\log(t/M+1)$, this yields
\begin{equation*}
    \t J_3 \gt F(t)^\g(t+M)^{1-\g}\log^{-\frac{\g}{2}}(t/M+1), \quad t\ge \t t_1.
\end{equation*}
Substituting this into \eqref{5.31} yields
\begin{align}
    F''(t) &\gt \frac\ve{t+M}, && t\ge0, \label{5.32}\\
    F''(t) &\gt F(t)^\g(t+M)^{-1-\g}\log^{-\frac{\g}{2}}(t/M+1), && t\ge \t t_1. \label{5.33}
\end{align}
Substituting $F(t)\gt \ve(t+M)\log(t/M+1)$ into \eqref{5.33} derives
\begin{equation*}
    F''(t) \gt \ve^\g(t+M)^{-1}\log^\frac{\g}{2}(t/M+1).
\end{equation*}
Integrating this yields
\begin{equation*}
    F(t) \gt \ve^\g(t+M)\log^\frac{\g+2}{2}(t/M+1).
\end{equation*}
Substituting this into \eqref{5.33} again gives
\begin{equation*}
    F''(t) \gt \ve^{\g^2}(t+M)^{-1}\log^\frac{\g(\g+1)}{2}(t/M+1)
    =\ve^{\g^2}(t+M)^{-1}\log^\frac{\g(\g^2-1)}{2(\g-1)}(t/M+1).
\end{equation*}
Repeating this process $k$ times, we see that
\begin{equation}\label{5.34}
    F''(t) \gt \ve^{\g^k}(t+M)^{-1}\log^\frac{\g(\g^k-1)}{2(\g-1)}(t/M+1),
\end{equation}
where $k=\left[\log_\g2\right]$. Solving \eqref{5.34} yields
\begin{equation*}
    F(t) \gt \ve^{\g^k}(t+M)\log^{\frac{\g(\g^k-1)}{2(\g-1)}+1}(t/M+1), \quad t\ge \t t_2,
\end{equation*}
where $\t t_2>0$ is a constant only depending on $\g$ and $M$.
Substituting this into \eqref{5.33} derives
\begin{equation}\label{5.35}
  F''(t) \gt F(t)\ve^{\g^k(\g-1)}(t+M)^{-2}\log^\frac{\g^{k+1}-2}{2}(t/M+1), \quad t\ge \t t_2,
\end{equation}
where $\frac{\g^{k+1}-2}{2}>0$ by the choice of $k=\left[\log_\g2\right]$.
Since \eqref{5.35} is analogous to \eqref{5.24}, as in Part~\RN{1}, we can choose $\t t_3\defeq O\Big(e^{C\ve^{-\frac{2\g^k(\g-1)}{\g^{k+1}-2}}}\Big)$ such that
\begin{equation*}
    F'(t) \gt \ve^\frac{\g^k(\g-1)}2(t+M)^{-1}\log^\frac{\g^{k+1}-2}{4}(t/M+1)\,F(t), \quad t\ge \t t_3,
\end{equation*}
which is similar to \eqref{5.28} and yields
\begin{equation}\label{5.36}
    F(t) \gt \ve^{C_\g}(t+M)^\frac{2(\g+2)}{\g-1}, \quad t\ge \t t_4\defeq C\t t_3^2,
\end{equation}
where $C_{\g}>0$ is a constant depending on $\g$.
Substituting \eqref{5.36} into \eqref{5.33} yields
\begin{equation}\label{5.37}
    F''(t) \gt \ve^{C_\g} F(t)^\frac{\g+1}2, \qquad t\ge \t t_4.
\end{equation}
Multiplying \eqref{5.37} by $F'(t)$ and integrating over the variable $t$ as in Part~\RN{1}, we have
\begin{equation*}
    F'(t) \gt \ve^{C_\g}F(t)^\frac{\g+3}4, \quad t\ge \t t_5\defeq C\t t_4.
\end{equation*}
Together with $\g>1$ and the choice of $\t t_3$, this yields $T_\ve<\infty$.

\smallskip

Collecting Part~\RN{1} and Part~\RN{2} completes the proof of Theorem~\ref{thm3}.
\end{proof}


\appendix
\section{Proof on the lower bound of $P(t,l)$ in $\Sigma\equiv \{(t,l)\colon
t\ge0, t+M_0\le l\le t+M\}$.}\label{appendix}

We fixed a point $A=(t_A,l_A)\in\Sigma$.
In the characteristic coordinates $\xi=1+t-l$ and $\zeta=1+t+l$, \eqref{5.5} can be written as
\begin{equation}\label{A.1}
    \mathscr{L}\bar P\defeq \p_{\xi\zeta}^2\bar P+\frac{2^{\la-2}\mu}{(\xi+\zeta)^\la}
    (\p_\xi\bar P+\p_\zeta\bar P)=\frac{\bar f}4,
\end{equation}
where $\bar P(\xi,\zeta)\defeq P(\frac{\zeta+\xi}2-1,\frac{\zeta-\xi}2)$.
The adjoint operator $\mathscr{L}^*$ of $\mathscr{L}$ has the form
\begin{equation}\label{A.2}
    \mathscr{L}^*\mcR\defeq \p_{\xi\zeta}^2\mcR-\frac{2^{\la-2}\mu}{(\xi+\zeta)^\la}
    (\p_\xi\mcR+\p_\zeta\mcR)+\frac{2^{\la-1}\mu\la}{(\xi+\zeta)^{\la+1}}\mcR.
\end{equation}
For the point $A=(\xi_A,\zeta_A)$ with $\xi_A+\zeta_A=2(1+t_A)\ge2$,
denote $B=(2-\zeta_A,\zeta_A)$, $C=(\xi_A,2-\xi_A)$ and $\mathscr{D}$,
the domain surrounded by the triangle $ABC$ (see Figure 1 below).

Let the numbers $a$ and $b$ satisfy $a+b=1$ and
\begin{equation*}
ab=\left\{
\begin{aligned}
  &\frac{\mu\la}{2}, && \la>1, \\
  &\frac\mu2(1-\frac\mu2), && \la=1.
\end{aligned}
\right.
\end{equation*}
We define
\begin{equation}\label{A.3}
    z\defeq -\frac{(\xi_A-\xi)(\zeta_A-\zeta)}{(\xi_A+\zeta_A)(\xi+\zeta)}
\end{equation}
and
\begin{equation}\label{A.4}
    \mcR(\xi,\zeta;\xi_A,\zeta_A)\defeq \Big[\frac{\Xi(\xi+\zeta-1)}{\Xi
    (\xi_A+\zeta_A-1)}\Big]^{2^{\la-2}} \Psi(a,b,1;z),
\end{equation}
here the definition of function $\Xi$ is given in \eqref{Xi-def} and $\Psi$ is the hypergeometric function.
\begin{figure}[htbp]
\centering\includegraphics[width=9cm,height=6.5cm]{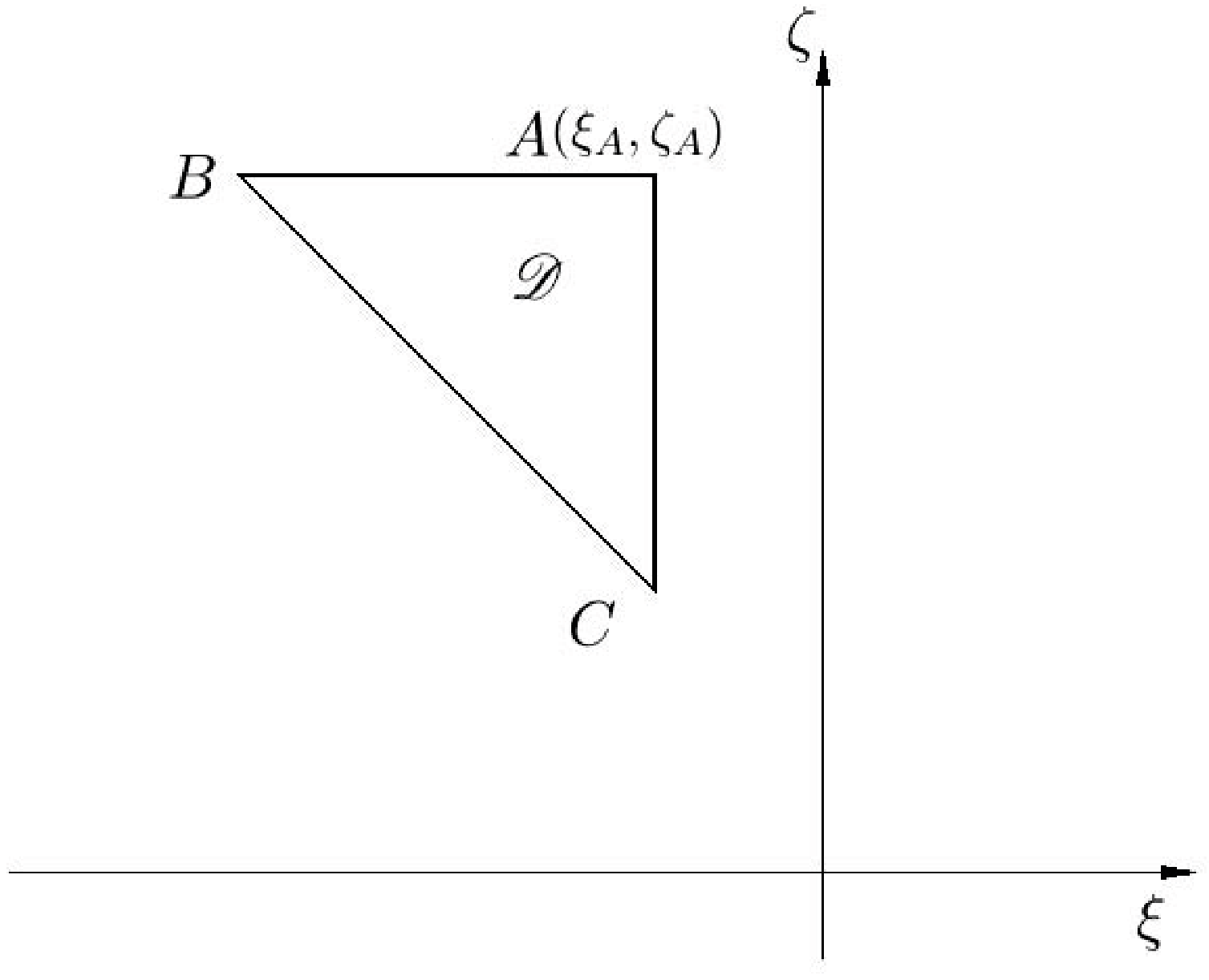}
\caption{\bf $(\xi, \zeta)-$plane}\label{fig:1}
\end{figure}
From this and direct calculation, we infer
\begin{equation}\label{A.5}
    \mathscr{L}^*\mcR=[\frac{2^{\la-2}\mu\la}{(\xi+\zeta)^{\la+1}}-
    \frac{ab}{(\xi+\zeta)^2}-\frac{4^{\la-2}\mu^2}{(\xi+\zeta)^{2\la}}]\mcR.
\end{equation}
On the other hand, from \eqref{A.1}-\eqref{A.2} we arrive at
\begin{equation*}
    \mcR\mathscr{L}\bar P-\bar P\mathscr{L}^*\mcR=\p_\zeta(\mcR\p_\xi\bar P
    +\frac{2^{\la-2}\mu}{(\xi+\zeta)^\la}\mcR\bar P)-\p_\xi(\bar P\p_\zeta\mcR
    -\frac{2^{\la-2}\mu}{(\xi+\zeta)^\la}\mcR\bar P).
\end{equation*}
Integrating this over $\mathscr{D}$ yields
\begin{align}\label{A.6}
    \bar P(A)&=\frac12\mcR(C;A)\bar P(C)+\frac12\mcR(B;A)\bar P(B)+\doubleint_\mathscr{D}
    (\mcR\mathscr{L}\bar P-\bar P\mathscr{L}^*\mcR)\,d\xi d\zeta \no\\
    &+\int_{BC}(\frac12\mcR\p_\xi\bar P-\frac12\bar P\p_\xi\mcR+\frac\mu4\mcR\bar P)\,d\xi
    +(\frac12\bar P\p_\zeta\mcR-\frac12\mcR\p_\zeta\bar P-\frac\mu4\mcR\bar P)\,d\zeta.
\end{align}
Returning to the variable $(t,l)$ (see Figure 2 below), we find in the second line of \eqref{A.6} that
\begin{align}\label{A.7}
    \int_{BC}\cdots=\int_B^C[\frac14\mcR(\p_t-\p_l)P-\frac14P(\p_t-\p_l)
     \mcR+\frac\mu4\mcR P]\,(-dl) \no\\
    +[\frac14P(\p_t+\p_l)\mcR-\frac14\mcR(\p_t+\p_l)P-\frac\mu4\mcR P]\,dl \no\\
    =\int_{l_A-t_A}^{l_A+t_A}\left.[\frac\mu2\mcR P+\frac12\mcR\p_tP
     -\frac12P\p_t\mcR]\right|_{t=0}dl \no\\
    =\int_{l_A-t_A}^{l_A+t_A} \Xi(t_A)^{-\frac12} \Big[\Psi(a,b,1;z|_{t=0})
     \Big(\frac\mu4q_0(l)+\frac12q_1(l)\Big) \no\\
    -\frac{ab}{2}\Psi(a+1,b+1,2;z|_{t=0})q_0(l)z_t|_{t=0}\Big]dl,
\end{align}
\begin{figure}[htbp]
\centering\includegraphics[width=8.5cm,height=6.5cm]{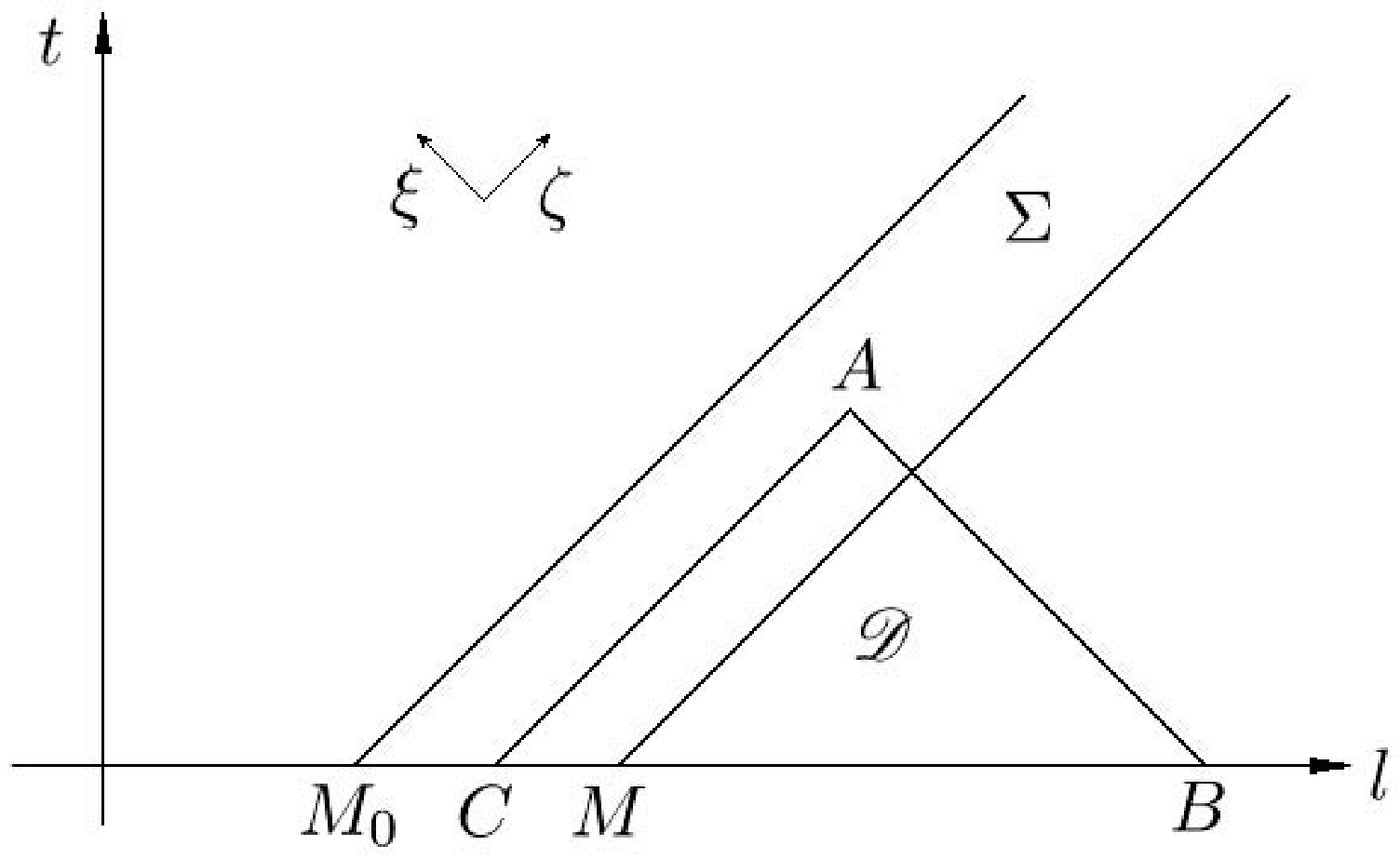}
\caption{\bf $(t, l)-$plane}\label{fig:2}
\end{figure}
where we have used the formula $\Psi'(a,b,c;z)=\frac{ab}{c}\Psi(a+1,b+1,c+1;z)$ (see page 58 of \cite{EMOT}).
From the definition \eqref{A.3}, we arrive at
\begin{equation*}
    z=-\frac{(t_A-l_A-t+l)(t_A+l_A-t-l)}{4(1+t_A)(1+t)}
\end{equation*}
and
\begin{equation}\label{A.8}
    z_t|_{t=0}=\frac{t_A}{2(1+t_A)}-z|_{t=0}.
\end{equation}
If $(t, l)\in\Sigma\cap\overline{\mathscr{D}}$, we infer
\begin{equation}\label{A.9}
0\ge z \ge -\frac12(M-M_0)\ge -\frac12\delta_0,
\end{equation}
which implies that \eqref{Psi-bound} holds.
This, together with \eqref{q0-positive}, \eqref{A.7}-\eqref{A.9} and the assumption \eqref{+condition} of $\Lambda \ge 3ab$,
yields that the integral in the second line of \eqref{A.6} is non-negative.
Next we prove that $P(t,l)\ge0$ for all $(t, l)\in\Sigma$. Define
\begin{equation*}
    \bar t\equiv\inf \{t\colon \exists~l\in(t+M_0,t+M)~s.t.~P(t,l)<0\}.
\end{equation*}
From assumption \eqref{q0-positive}, we get $\bar t>0$.
If $\bar t<+\infty$, we see that there exists $\bar l\in(\bar t+M_0,\bar t+M)$ such that $P(\bar t,\bar l)=0$.
Moreover, we have $P(t,l)\ge0$ for $t<\bar t$.
Choose $A=(t_A,l_A)=(\bar t,\bar l)$ in \eqref{A.6}.
From \eqref{A.4}-\eqref{A.5} and \eqref{Psi-bound} we infer $\mathscr{L}^*\mcR\le0$ for $\la\ge1$ and $(t,l)\in\Sigma\cap\mathscr{D}$ ($\mathscr{L}^*\mcR\equiv0$ if $\la=1$).
It follows from $f(t,l)\ge0$ in \eqref{5.5}, \eqref{Psi-bound}-\eqref{q1-positive} and \eqref{A.6} that
\begin{align*}
    P(\bar t,\bar l)\ge \frac12\mcR(C;A)P(0,\bar l-\bar t)+\doubleint_{\Sigma\cap\mathscr{D}}
    (\mcR\mathscr{L}\bar P-\bar P\mathscr{L}^*\mcR)\,d\xi d\zeta \ge \frac14\Xi(\bar t)^{-\frac12}q_0(\bar l-\bar t)>0,
\end{align*}
which is a contradiction with $P(\bar t,\bar l)=0$.
Consequently, we conclude that $\bar t=+\infty$ and $P(t,l)\ge0$ for all $(t, l)\in\Sigma$.
It follows from \eqref{Psi-bound}-\eqref{q1-positive}, \eqref{A.4}, \eqref{A.6}, $P(t,l)\ge0$ and $\mathscr{L}^*\mcR\le0$ that
\begin{equation*}
    P(t_A,l_A)\ge \frac14\Xi(t_A)^{-\frac12}q_0(l_A-t_A)+\frac14\int_0^{t_A}\int_{l_A-t_A+\tau}^{l_A+t_A-\tau} \left(\frac{\Xi(\tau)}{\Xi(t_A)}\right)^\frac12 f(\tau,y)\,dyd\tau,
\end{equation*}
which is \eqref{5.8}.

\medskip

\begin{acknowledgement}
Yin Huicheng wishes to express his gratitude to Professor Ingo Witt,
 University of G\"ottingen, and Professor Michael
Reissig, Technical University Bergakademie Freiberg, for their
interests in this problem and some very fruitful discussions in the past.
\end{acknowledgement}



\end{document}